\documentclass[11pt]{amsart}

\usepackage{amsmath}
\usepackage{amsfonts}
\usepackage{amssymb}
\usepackage{amsthm}
\usepackage{graphicx}
\usepackage{pb-diagram}
\usepackage{epstopdf}
\usepackage{amscd}
\usepackage{pdfpages}
\usepackage{bbm}
\usepackage{multirow}
\usepackage[mathscr]{eucal}
\usepackage{epsfig,epsf,epic}
\usepackage{pdfsync}
\usepackage{enumerate}
\usepackage{etex}
\usepackage[all]{xy}
\usepackage{fancyref}
\usepackage{mathtools}
\usepackage{scalerel,stackengine}
\usepackage{comment}
\usepackage{hyperref}

\stackMath

\newtheorem{theorem}{Theorem}[section]

\newtheorem{corollary}[theorem]{Corollary}
\newtheorem{definition}[theorem]{Definition}

\newtheorem{lemma}[theorem]{Lemma}
\newtheorem{remark}[theorem]{Remark}
\newtheorem{assum}[theorem]{Assumption}
\newtheorem{condition}[theorem]{Condition}
\newtheorem{notation}[theorem]{Notation}
\numberwithin{equation}{section}

\DeclareMathOperator{\trace}{tr}

\newcommand{\real}{\mathbb{R}}
\newcommand{\comp}{\mathbb{C}}

\newcommand{\inte}{\mathbb{Z}}

\newcommand{\pdb}{\bar{\partial}}

\newcommand{\half}{\frac{1}{2}}

\newcommand{\reallywidehat}[1]{%
	\savestack{\tmpbox}{\stretchto{%
			\scaleto{%
				\scalerel*[\widthof{\ensuremath{#1}}]{\kern-.6pt\bigwedge\kern-.6pt}%
				{\rule[-\textheight/2]{1ex}{\textheight}}
			}{\textheight}%
		}{0.5ex}}%
	\stackon[1pt]{#1}{\tmpbox}%
}

\newcommand{\reallywidetilde}[1]{%
	\savestack{\tmpbox}{\stretchto{%
			\scaleto{%
				\scalerel*[\widthof{\ensuremath{#1}}]{\kern-.6pt\sim\kern-.6pt}%
				{\rule[-\textheight/2]{1ex}{\textheight}}
			}{\textheight}%
		}{0.5ex}}%
	\stackon[1pt]{#1}{\tmpbox}%
}

\newcommand{\der}[3]{\text{Der}^{#3}_{#1}(#2,#2)}
\newcommand{\anchor}[1]{\prescript{#1}{}{\alpha}}
\newcommand{\morph}[1]{\prescript{#1}{}{f}}
\newcommand{\morphh}[1]{\prescript{#1}{}{\hat{f}}}
\newcommand{\gla}[1]{\prescript{#1}{}{\widetilde{\mathcal{G}}}}
\newcommand{\gluesset}{\mathfrak{G}}
\newcommand{\gluessetp}{\widetilde{\mathfrak{G}}}
\newcommand{\gluediff}{\mathbf{G}}
\newcommand{\gluediffp}{\tilde{\mathbf{G}}}
\newcommand{\mcsimplicial}[1]{\prescript{#1}{}{\mathfrak{M}\mathfrak{C}}}
\newcommand{\mcsimplicialp}[1]{\prescript{#1}{}{\widetilde{\mathfrak{M}\mathfrak{C}}}}
\newcommand{\mcsimplicialh}[1]{\prescript{#1}{}{\widehat{\mathfrak{M}\mathfrak{C}}}}
\newcommand{\twcp}[1]{\prescript{#1}{}{\widetilde{TW}}}

\newcommand{\tf}[1]{\prescript{#1}{}{TW\mathfrak{F}}}
\newcommand{\two}[1]{\prescript{#1}{}{TW\mathcal{O}}}
\newcommand{\restp}[1]{\prescript{#1}{}{\tilde{\flat}}}
\newcommand{\patchp}[1]{\prescript{#1}{}{\tilde{\psi}}}
\newcommand{\patchh}[1]{\prescript{#1}{}{\hat{\psi}}}
\newcommand{\restap}[1]{\prescript{#1}{}{\tilde{\mathfrak{b}}}}
\newcommand{\patchijp}[1]{\prescript{#1}{}{\tilde{\mathfrak{p}}}}
\newcommand{\cocyobsp}[1]{\prescript{#1}{}{\tilde{\mathfrak{o}}}}
\newcommand{\gluep}[1]{\prescript{#1}{}{\tilde{g}}}
\newcommand{\iautop}[1]{\prescript{#1}{}{\tilde{\mathfrak{a}}}}
\newcommand{\sautop}[1]{\prescript{#1}{}{\tilde{\vartheta}}}
\newcommand{\autoijp}[1]{\prescript{#1}{}{\tilde{\phi}}}
\newcommand{\dbtwistp}[1]{\prescript{#1}{}{\tilde{\mathfrak{d}}}}
\newcommand{\kla}{\prescript{0}{}{\mathcal{R}}}
\newcommand{\shom}{\mathscr{H}om}
\newcommand{\sla}[1]{\prescript{#1}{}{\widehat{\mathcal{G}}}}

\newcommand{\restah}[1]{\prescript{#1}{}{\hat{\mathfrak{b}}}}
\newcommand{\patchijh}[1]{\prescript{#1}{}{\hat{\mathfrak{p}}}}
\newcommand{\cocyobsh}[1]{\prescript{#1}{}{\hat{\mathfrak{o}}}}

\newcommand{\bchprod}{\odot}
\newcommand{\blat}{\mathbf{K}}
\newcommand{\cfr}{R}
\newcommand{\cfrk}[1]{\prescript{#1}{}{\cfr}}
\newcommand{\logs}{S^{\dagger}}
\newcommand{\logsk}[1]{\prescript{#1}{}{S}^{\dagger}}
\newcommand{\logsf}{\hat{S}^{\dagger}}

\newcommand{\bva}[1]{\prescript{#1}{}{\mathcal{G}}}

\newcommand{\rest}[1]{\prescript{#1}{}{\flat}}
\newcommand{\patch}[1]{\prescript{#1}{}{\psi}}
\newcommand{\resta}[1]{\prescript{#1}{}{\mathfrak{b}}}
\newcommand{\patchij}[1]{\prescript{#1}{}{\mathfrak{p}}}
\newcommand{\cocyobs}[1]{\prescript{#1}{}{\mathfrak{o}}}

\newcommand{\simplex}{\blacktriangle}
\newcommand{\simplexbdy}{\vartriangle}
\newcommand{\hsimplex}{\blacksquare}
\newcommand{\hsimplexbdy}{\square}
\newcommand{\twc}[1]{\prescript{#1}{}{TW}}
\newcommand{\restmap}{\mathfrak{r}}
\newcommand{\glue}[1]{\prescript{#1}{}{g}}
\newcommand{\iauto}[1]{\prescript{#1}{}{\mathfrak{a}}}
\newcommand{\sauto}[1]{\prescript{#1}{}{\vartheta}}
\newcommand{\autoij}[1]{\prescript{#1}{}{\phi}}

\newcommand{\cech}[1]{\prescript{#1}{}{\check{\mathcal{C}}}}
\newcommand{\cechd}[1]{\prescript{#1}{}{\delta}}
\newcommand{\dbtwist}[1]{\prescript{#1}{}{\mathfrak{d}}}
\newcommand{\polyv}[1]{\prescript{#1}{}{\mathscr{L}}}

\newcommand{\glueh}[1]{\prescript{#1}{}{\hat{g}}}
\newcommand{\iautoh}[1]{\prescript{#1}{}{\hat{\mathfrak{a}}}}
\newcommand{\sautoh}[1]{\prescript{#1}{}{\hat{\vartheta}}}
\newcommand{\morphc}[1]{\prescript{#1}{}{h}}
\newcommand{\dbtwisth}[1]{\prescript{#1}{}{\hat{\mathfrak{d}}}}

\begin{document}
\title[Smoothing pairs over degenerate CY]{Smoothing pairs over degenerate Calabi-Yau varieties}	
	
\author[Chan]{Kwokwai Chan}
\address{Department of Mathematics\\ The Chinese University of Hong Kong\\ Shatin\\ Hong Kong}
\email{kwchan@math.cuhk.edu.hk}	
	
\author[Ma]{Ziming Nikolas Ma}
\address{The Institute of Mathematical Sciences and Department of Mathematics\\ The Chinese University of Hong Kong\\ Shatin \\ Hong Kong}
\email{zmma@math.cuhk.edu.hk}

\begin{abstract}
	We apply the techniques developed in \cite{chan2019geometry} to study smoothings of a pair $(X,\mathfrak{C}^*)$, where $\mathfrak{C}^*$ is a bounded perfect complex of locally free sheaves over a degenerate Calabi-Yau variety $X$.
	In particular, if $X$ is a projective Calabi-Yau variety admitting the structure of a toroidal crossing space and with the higher tangent sheaf $\mathcal{T}^1_X$ globally generated, and $\mathfrak{F}$ is a locally free sheaf over $X$, then we prove, using the results in \cite{Felten-Filip-Ruddat}, that the pair $(X,\mathfrak{F})$ is formally smoothable when $\text{Ext}^2(\mathfrak{F},\mathfrak{F})_0 = 0$ and $H^2(X,\mathcal{O}_X) = 0$. 
\end{abstract}

\maketitle

\section{Introduction}\label{sec:introduction}

\subsection{Background}\label{sec:intro_background}
After pioneering works of Quillen, Deligne and Drinfeld, it is now a universally accepted philosophy that any deformation problem over a field of characteristic zero should be governed by the Maurer-Cartan equation of a {\em differential graded Lie algebra} (abbreviated as {\em dgLa}) or {\em $L_{\infty}$-algebra}. Lots of works have been done in this direction; see e.g. \cite{fiorenza2012differential,  fiorenza2007structures, fiorenza2012formality, fiorenza2012cosimplicial,  iacono2007differential, iacono2014deformations, KKP08, kontsevichgeneralized, Lurie, manetti2007lie, manetti2005differential, manetti2015some, pridham2010unifying}. In many cases, the existence of an underlying dgLa facilitates the use of algebraic techniques in solving the relevant geometric deformation problem.

An important problem of such kind is the deformation theory of a pair $(X,\mathfrak{F})$, where $X$ is an algebraic variety and $\mathfrak{F}$ is a coherent sheaf over $X$.
When $X$ is smooth, this problem has been studied in detail in \cite{huang1995joint, huybrechts2010deformation, li2008deformation, martinengohigher, sernesi2007deformations}, and the approach using dgLa was also carried out in \cite{chan2016differential} when $\mathfrak{F}$ is locally free and in \cite{iacono2018deformations} when $\mathfrak{F}$ is a general coherent sheaf. 
For singular $X$, the only known case seems to be when $X$ is a reduced local complete intersection and $\mathfrak{F}$ is a line bundle as studied by Wang \cite{wang2012deformations}.
On the other hand, when $X$ is a maximally degenerate Calabi-Yau variety and $\mathfrak{L}$ is an ample line bundle over $X$, deformations of the pair $(X, \mathfrak{L})$ are closely related to the study of theta functions on smoothings of $X$ in \cite{GHKS2016theta, gross2016theta}.
Very little is known when $X$ is singular and $\mathfrak{F}$ is of higher rank. From the perspective of mirror symmetry, a thorough understanding of the deformation theory of $(X,\mathfrak{F})$, where $X$ is at a large complex structure limit and $\mathfrak{F}$ is a coherent sheaf over $X$, is desirable for investigating the correspondence between B-branes on $X$ and A-branes on the mirror.

A major difficulty in obtaining smoothings of a singular variety is that there is always nontrivial topology change in a degeneration while classical deformation theories can only produce equisingular deformations.  
In our recent joint work \cite{chan2019geometry} with Leung, we discovered that this difficulty could be overcome by using Thom-Whitney simplicial constructions to build an {\em almost dgLa}
$\polyv{}^{*}(X)$ over $\comp[[q]]$, instead of a genuine dgLa, from a degenerate Calabi-Yau variety $X$ equipped with local thickening data. The {\em almost} condition here means that the differential of $\polyv{}^{*}(X)$ squares to $0$ only modulo $q$ -- this property reflects precisely the fact that locally trivial deformations are not allowed in a smoothing.\footnote{What we call an almost dgLa here and in \cite{chan2019geometry} is called a {\em pre-dgLa} in Felten's very recent paper \cite{Felten}, where he also explained why ordinary dgLa's do not suffice for the purpose of smoothing singular (even log smooth) varieties.} This can be regarded as providing a singular version of the extended Kodaira-Spencer dgLa $\Omega^{0,*}(X, \wedge^*T_X)[[q]]$ for smooth $X$.

We prove in \cite{chan2019geometry} that smoothings of $X$ are indeed governed by the Maurer-Cartan equation of $\polyv{}^{*}(X)$, which, under certain assumptions, can be solved by general algebraic techniques \cite{KKP08, kontsevichgeneralized, terilla2008smoothness}. This yields Bogomolov-Tian-Todorov--type unobstructedness theorems \cite{bogomolov1978hamiltonian, tian1987smoothness, todorov1989weil} and hence smoothness of the extended moduli space.
Our abstract algebraic framework was applied very recently by Felten-Filip-Ruddat in \cite{Felten-Filip-Ruddat} to obtain smoothings of a very general class of varieties which includes the log smooth Calabi-Yau varieties studied by Friedman \cite{Friedman83} and Kawamata-Namikawa \cite{kawamata1994logarithmic} as well as the maximally degenerate Calabi-Yau varieties studied by Kontsevich-Soibelman \cite{kontsevich-soibelman04} and Gross-Siebert \cite{Gross-Siebert-logI, Gross-Siebert-logII, gross2011real}. 

The goal of this paper is to extend the techniques in \cite{chan2019geometry} to study smoothings of a pair $(X,\mathfrak{F})$, where $\mathfrak{F}$ is a locally free sheaf on a degenerate Calabi-Yau variety $X$. 
Our main result is that, given local thickenings of $\mathfrak{F}$ along with the local thickenings of $X$, there exists an almost dgLa $\polyv{}^*(X,\mathfrak{F})$ which governs the smoothing of $(X,\mathfrak{F})$.
We then apply this to the class of smoothable degenerate Calabi-Yau varieties $X$ obtained in \cite{Felten-Filip-Ruddat}.
Under the further assumptions that
$\text{Ext}^2(\mathfrak{F},\mathfrak{F})_0 = 0$ and that the pair $(X, \det \mathfrak{F})$ (where $\det \mathfrak{F}$ denotes the determinant line bundle of $\mathfrak{F}$) is formally smoothable, our results show that the pair $(X,\mathfrak{F})$ is also formally smoothable.

In the subsequent joint work \cite{chan2020tropical} with Suen, this smoothability result will be applied to a pair $(X, \mathfrak{F})$, where $X$ is a maximally degenerate K3 surface and $\mathfrak{F}$ is a locally free sheaf of arbitrary rank over $X$ associated to a so-called {\em tropical Lagrangian multi-section}, which should arise as a tropical limit of Lagrangian multi-sections in an SYZ fibration of the mirror (cf. \cite{fukaya05, LYZ}). We devise a combinatorial criterion for checking the condition that $\text{Ext}^2(\mathfrak{F},\mathfrak{F})_0 = 0$, thus producing new explicit examples of smoothable pairs. 

\subsection{Main results}\label{sec:intro_main_result}

Before explaining the main result of this paper, let us first state our major geometric application.
\begin{theorem}\label{thm:intro_main_application}
	Let $X$ be a projective toroidal crossing space which is Calabi-Yau (in the sense that $\omega_X \cong \mathcal{O}_X$) and satisfies the assumption that the higher tangent sheaf $\mathcal{T}^1_X := \mathcal{E}\text{xt}^1(\Omega_X,\mathcal{O}_X)$ is globally generated, and $\mathfrak{F}$ be a locally free sheaf over $X$.
	Then the pair $(X, \mathfrak{F})$ is formally smoothable when $\text{Ext}^2(\mathfrak{F},\mathfrak{F})_0  = 0$ and the pair $(X, \det \mathfrak{F})$ is formally smoothable. In particular, $(X, \mathfrak{F})$ is always formally smoothable if $H^2(X,\mathcal{O}_X) = 0$ and $\text{Ext}^2(\mathfrak{F},\mathfrak{F})_0  = 0$.
\end{theorem}
Some explanations of this statement are in order:
\begin{itemize}
	\item 
	A {\em toroidal crossing space} is defined in \cite[Definition 1.5]{Felten-Filip-Ruddat} (see also the 2nd paragraph of \S \ref{sec:deformation_of_pair_on_log}). A typical example is given by a normal crossing space, i.e., a connected variety locally of the form $z_0\cdots z_k = 0$ for $(z_0, \dots, z_n)\in \comp^{n+1}$ where $n = \dim X$. The boundary divisor in a Gorenstein toric variety is also naturally a toroidal crossing space. 
	\item We say a space or a pair is {\em formally smoothable} if there exists a compatible system of thickenings over $\comp[q]/(q^{k+1})$ for $k \in \mathbb{N}$; see Definition \ref{def:formally_smoothable} for the precise definition.
	\item
	(Formal) smoothability of a projective toroidal crossing space which is Calabi-Yau and such that $\mathcal{T}^1_X$ is globally generated follows from the main results in \cite{Felten-Filip-Ruddat}. Note that the condition here is weaker than Friedman's famous notion of $d$-semistability \cite{Friedman83} which requires that $\mathcal{T}^1_X$ is trivial. Theorem \ref{thm:intro_main_application} can be viewed as giving a sufficient condition for formal smoothability of the pair $(X, \mathfrak{F})$ under the assumption that $X$ is formally smoothable.
	\item
	More generally, we may let $X$ be a projective toroidal crossing space which is Calabi-Yau and satisfies the assumptions in \cite[Theorem 1.7]{Felten-Filip-Ruddat}. Felten-Filip-Ruddat \cite{Felten-Filip-Ruddat} equipped such $X$ with a log structure locally modeled by the same types of potentially singular log schemes that appeared in the Gross-Siebert program \cite{Gross-Siebert-logI, Gross-Siebert-logII}, and proved the degeneracy of the Hodge-to-de Rham spectral sequence for the log de Rham complex at the $E_1$ page. This enables them to apply the dgBV framework and main results in \cite{chan2019geometry} to prove that $X$ is smoothable to an orbifold with terminal singularities. Theorem \ref{thm:intro_main_application} applies to such varieties as well.
\end{itemize}

Theorem \ref{thm:intro_main_application} follows from a more general result concerning the formal smoothability of a pair $(X, \mathfrak{C}^*)$, where $X$ is again a projective toroidal crossing space which is Calabi-Yau and $\mathfrak{C}^*$ is a bounded complex of locally free sheaves (of finite rank) over $X$. We will extend and apply the abstract algebraic framework in \cite{chan2019geometry} to study this smoothability problem. 
Throughout this paper, we work over $\comp$ and with the polynomial rings $\cfr := \comp[q]$ and $\cfrk{k} := \comp[q]/(q^{k+1})$.

Before describing our main result, let us recall the main ideas behind the construction in \cite{chan2019geometry}. Our starting point was a degenerate Calabi-Yau variety $X$, equipped with a covering $\mathcal{V} = \{V_\alpha\}_{\alpha}$ by Stein open subsets together with a local $k$-th order thickening $\prescript{k}{}{\mathbf{V}}_{\alpha}$ over $\cfrk{k}$ for each $\alpha$ and $k \in \mathbb{N}$. From the local model $\prescript{k}{}{\mathbf{V}}_{\alpha}$, we obtain a coherent sheaf $\bva{k}_{\alpha}^*$ of Batalin-Vilkovisky (BV) algebras\footnote{Or Gerstenhaber algebras, as advocated in \cite{Felten-Filip-Ruddat}.} over $\cfrk{k}$ on $V_\alpha$. 
We further assume that there is another Stein open covering $\mathcal{U} = \{U_i\}_{i \in \mathbb{N}}$ forming a basis of the topology and a biholomorphism $\prescript{k}{}{\Psi}_{\alpha\beta,i} : \prescript{k}{}{\mathbf{V}}_{\alpha}|_{U_i} \rightarrow \prescript{k}{}{\mathbf{V}}_{\beta}|_{U_i}$ for each triple $(U_i; V_\alpha, V_\beta)$ such that $U_i \subset V_\alpha \cap V_\beta$. In geometric situations, these {\em higher order patching data} always arise from uniqueness of the local models $\prescript{k}{}{\mathbf{V}}_{\alpha}$'s.

However, the biholomorphisms $\prescript{k}{}{\Psi}_{\alpha\beta,i}$'s do not satisfy the cocycle condition on $U_i \subset V_{\alpha\beta\gamma}:=V_{\alpha}\cap V_{\beta} \cap V_{\gamma}$, so we cannot simply glue the sheaves $\bva{k}_{\alpha}^*$'s together. Fortunately, the discrepancies are always captured by the exponential action of local sections of $\bva{k}^{*}_{\alpha}$'s. The key idea in \cite{chan2019geometry} is to consider the Thom-Whitney simplicial construction \cite{dupont1976simplicial, whitney2012geometric}, giving rise to a dg resolution $\polyv{k}^{*}_{\alpha}$ (as a sheaf of dgBV algebras) of $\bva{k}^{*}_{\alpha}$, which can be regarded as a simplicial replacement of the Dolbeault resolution. The upshot is that these sheaves $\polyv{k}^{*}_{\alpha}$'s of dgBV algebras, upon suitable modifications of the various operators like the differential and BV operator, can then be glued together to produce a global sheaf over $X$ whose global sections give an {\em almost dgBV algebra} $\polyv{}^{*}(X)$ (meaning that its differential squares to 0 only modulo $q$). Morally speaking, the sheaves $\polyv{k}^{*}_{\alpha}$'s can be glued because they are softer than the sheaves $\bva{k}^{*}_{\alpha}$'s.

To extend this construction to the case of a pair $(X, \mathfrak{C}^*)$, we just need to prescribe a local model for how the singular variety $X$ along with $\mathfrak{C}^*$ are being smoothed out over $\comp[q]$.
So we first assume that there is a local $k$-th order thickening $(\prescript{k}{}{\mathbf{V}}_{\alpha},\prescript{k}{}{\mathfrak{C}}^*_{\alpha})$ of each $(V_{\alpha},\mathfrak{C}^*|_{V_{\alpha}})$ over $\cfrk{k}$ for any $k \in \mathbb{N}$. This so-called {\em geometric lifting datum} (see Definition \ref{def:geometric_lifting_data_for_pair}) gives rise to a coherent sheaf $\gla{k}^{*}_{\alpha}$ of dgLa's (which would be a coherent sheaf of Lie algebras if $\mathfrak{C}^* = \mathfrak{F}$ is a complex concentrated in degree 0).
We further assume that there is a biholomorphism $\prescript{k}{}{\Psi}_{\alpha\beta,i} : \prescript{k}{}{\mathbf{V}}_{\alpha}|_{U_i} \rightarrow \prescript{k}{}{\mathbf{V}}_{\beta}|_{U_i}$ and a sheaf isomorphism $\prescript{k}{}{\Xi}_{\alpha\beta,i}: \prescript{k}{}{\mathfrak{C}}^*_{\alpha}|_{U_i} \rightarrow \prescript{k}{}{\mathfrak{C}}_{\beta}^*|_{U_i}$ compatible with $\prescript{k}{}{\Psi}_{\alpha\beta,i}$ for each triple $(U_i; V_\alpha, V_\beta)$ such that $U_i \subset V_\alpha \cap V_\beta$: 
$$
\xymatrix@1{ \prescript{k}{}{\mathfrak{C}}^*_{\alpha}|_{U_i} \ar[r]^{\prescript{k}{}{\Xi}_{\alpha\beta,i}} \ar[d]^{} & \prescript{k}{}{\mathfrak{C}}_{\beta}^*|_{U_i} \ar[d]^{}\\
	\prescript{k}{}{\mathbf{V}}_{\alpha}|_{U_i} \ar[r]^{\prescript{k}{}{\Psi}_{\alpha\beta,i}} & \prescript{k}{}{\mathbf{V}}_{\beta}|_{U_i}.}
$$
This so-called {\em geometric patching datum} (see Definition \ref{def:geometric_patching_data_of_pair}) gives (non-canonical) identifications of the local models $(\prescript{k}{}{\mathbf{V}}_{\alpha},\prescript{k}{}{\mathfrak{C}}^*_{\alpha})$ when their supports overlap.
As in \cite{chan2019geometry}, we do not require these patching data to satisfy the cocycle condition on $U_i \subset V_{\alpha\beta\gamma}:=V_{\alpha}\cap V_{\beta} \cap V_{\gamma}$ (which is the case in geometric situations) and be compatible for different $k$'s at this stage. Rather, we assume that the discrepancies are captured by the exponential action of local sections of the sheaves $\gla{k}^{*}_{\alpha}$'s.

 
We can then apply the same technique as in \cite{chan2019geometry} (with a weakened condition to be explained in \S \ref{sec:weaken_assumption}). Namely, by passing to the dg resolution $\twcp{k}^*_{\alpha}$ of $\gla{k}_{\alpha}$ for each $\alpha$, we can glue these local sheaves together to obtain a global sheaf whose global sections produce an almost dgLa $\polyv{k}^*(\gluep{},\mathfrak{C}^*)$ whose differential $\prescript{k}{}{\pdb} + \prescript{k}{}{d}$ (where $\prescript{k}{}{d}$ is the differential acting on $\bva{k}_{\alpha}^*$ that we constructed in \cite{chan2019geometry} and mentioned above) squares to $0$ only modulo $q$. The associated Maurer-Cartan equation
\begin{equation}\label{eqn:intro_mc_equation}
(\prescript{k}{}{\pdb} + \prescript{k}{}{d}) \prescript{k}{}{\varphi} + \half [ \prescript{k}{}{\varphi}, \prescript{k}{}{\varphi}] =0,
\end{equation}
where $[\cdot,\cdot]$ denotes the Lie bracket of $\polyv{k}^*(\gluep{},\mathfrak{C}^*)$, governs formal smoothings of the pair $(X,\mathfrak{C}^*)$.

The main result of this paper is as follows.
\begin{theorem}\label{thm:intro_main_theorem}
	Let $X$ be as in Theorem \ref{thm:intro_main_application} and $\mathfrak{C}^*$ be a bounded complex of locally free sheaves on $X$ equipped with local thinkening data as described in Definitions \ref{def:geometric_lifting_data_for_pair} and \ref{def:geometric_patching_data_of_pair}.
	If $\text{Ext}^2(\mathfrak{C}^*,\mathfrak{C}^*)_0  = 0$ and the pair $(X, \det \mathfrak{C}^*)$ is formally smoothable, then there exists a system of Maurer-Cartan elements $\{\prescript{k}{}{\varphi} \in \polyv{k}^*(\gluep{},\mathfrak{C}^*)\}_{k \in \mathbb{N}}$ such that $\prescript{k+1}{}{\varphi} \equiv \prescript{k}{}{\varphi} \ (\text{mod $q^{k+1}$})$. 
	In particular, if $H^2(X,\mathcal{O}_X) = 0$ and $\text{Ext}^2(\mathfrak{C}^*,\mathfrak{C}^*)_0  = 0$, then the same conclusion holds.
\end{theorem}
 

In the case when $\mathfrak{C}^*$ is just one locally free sheaf $\mathfrak{F}$ concentrated in degree $0$, we can explicitly construct a geometric lifting datum (Definition \ref{def:geometric_lifting_data_for_pair}) and a geometric patching data (Definition \ref{def:geometric_patching_data_of_pair}) by trivializing $\mathfrak{F} = \bigoplus_{i=1}^r \mathcal{O}_{X}|_{V_{\alpha}} \cdot e_i$ and taking $\prescript{k}{}{\mathfrak{F}}_{\alpha} := \bigoplus_{i=1}^r \prescript{k}{}{\mathcal{O}}_{\alpha} \cdot e_i$ on a Stein open subset $V_{\alpha}$; see the 2nd paragraph of \S \ref{sec:geometric_gluing_from_mc} for details.
To prove Theorem \ref{thm:intro_main_application}, we construct
a geometric smoothing of $\mathfrak{F}$ from the Maurer-Cartan solution $\varphi = (\prescript{k}{}{\varphi})_{k\in \mathbb{N}}$ obtained in Theorem \ref{thm:intro_main_theorem} as follows. We take the dg resolution $\tf{k}^*_{\alpha}$ of the complex $\prescript{k}{}{\mathfrak{F}}_{\alpha}$ as a sheaf of complexes, and use the gluing morphisms $\gluep{k}_{\alpha\beta} : \tf{k}^*_{\alpha} \rightarrow \tf{k}^*_{\beta}$ (induced by gluing morphisms $\gluep{k}_{\alpha\beta}$ above) to produce a global sheaf of complexes $\prescript{k}{}{\mathfrak{F}}^*(\gluep{})$ with differential $\prescript{k}{}{\pdb}$ (notice that we do not have $\prescript{k}{}{d}$ because $\mathfrak{C}^*$ is concentrated in degree $0$) which squares to $0$ only modulo $q$. The Maurer-Cartan solution $\prescript{k}{}{\varphi}$, via the natural action of $\gla{k}_{\alpha}$ on $\prescript{k}{}{\mathfrak{F}}_{\alpha}$, gives a differential $\prescript{k}{}{\pdb}  + \prescript{k}{}{\varphi} \cdot$ for $\prescript{k}{}{\mathfrak{F}^*}(\gluep{})$ which squares to $0$ honestly. The cohomology sheaf $\prescript{k}{}{\mathfrak{F}}:=H^0(\prescript{k}{}{\mathfrak{F}}^*(\gluep{}) , \prescript{k}{}{\pdb}+ \prescript{k}{}{\varphi} \cdot)$ then gives a $k$-th order thickening of $\mathfrak{F}$ for each $k \in \mathbb{N}$, which are compatible in the sense that $\prescript{k+1}{}{\mathfrak{F}} \otimes_{(\cfrk{k+1})} \cfrk{k}= \prescript{k}{}{\mathfrak{F}}$. This produces the desired formal smoothing of $\mathfrak{F}$ of the pair $(X, \mathfrak{F})$, and hence proves Theorem \ref{thm:intro_main_application}. See \S \ref{sec:geometric_gluing_from_mc} for more details.


\subsection{Outline of the paper}
This paper is organized as follows.
In \S \ref{sec:abstract_algebra_for_deformation_of_pairs}, we recall the necessary abstract algebra from \cite{iacono2018deformations} for defining the dgLa which controls deformations of a pair.
In \S \ref{sec:gluing_construction}, we develop the abstract algebraic framework for constructing the almost dgLa $\polyv{k}^*(\gluep{},\mathfrak{C}^*)$ from prescribed abstract local deformation data, following the approach in \cite{chan2019geometry}. The main result is Theorem \ref{thm:main_theorem}, which is a general theorem on smoothness of Maurer-Cartan functors in this abstract setting.
In \S \ref{sec:deformation_of_pair_on_log}, we first review how to obtain the abstract local deformation data for a singular variety $X$ in \S \ref{sec:data_from_gross_siebert}. Then the construction of such data for a pair $(X, \mathfrak{C}^*)$ and the proof of Theorem \ref{thm:intro_main_theorem} using Theorem \ref{thm:main_theorem} are given in \S \ref{sec:pair_data_from_gross_siebert}. Finally, in \S \ref{sec:geometric_gluing_from_mc}, we investigate how to obtain a geometric smoothing of the pair $(X,\mathfrak{F})$, where $\mathfrak{F}$ is a locally free sheaf over $X$, and proves Theorem \ref{thm:intro_main_application}. 

\section*{Acknowledgement}
We would like to thank Simon Felten, Helge Ruddat and Yat-Hin Suen for very useful discussions, comments and suggestions.
We are also grateful to the anonymous referees for insightful comments and suggestions which helped to correct an error and largely improve the exposition of this paper.
K. Chan was supported by a grant of the Hong Kong Research Grants Council (Project No. CUHK14302617) and direct grants from CUHK. Z. N. Ma was partially supported by the Institute of Mathematical Sciences (IMS) and the Department of Mathematics at The Chinese University of Hong Kong.
\section*{Notation Summary}

\begin{notation}\label{not:bigrading_notation}
	For a $\inte^2$-graded vector space $V^{*,*} = \bigoplus_{p,q}V^{p,q}$, we write $V^{k} := \bigoplus_{p+q= k} V^{p,q}$, and $V^{*} := \bigoplus_{k} V^{k}$. We will simply write $V$ if we do not need to emphasize the gradings.
\end{notation}

\begin{notation}\label{not:universal_monoid}
	We fix a rank $s$ lattice $\blat$ together with a strictly convex $s$-dimensional rational polyhedral cone $Q_\real \subset \blat_\real:= \blat\otimes_\inte \real$. We let $Q := Q_\real \cap \blat$ and call it the {\em universal monoid}. We consider the ring $\cfr:=\comp[Q]$ and write a monomial element as $q^m \in \cfr$ for $m \in Q$, and consider the maximal ideal given by $\mathbf{m}:= \comp[Q\setminus \{0\}]$. 
	We let $\cfrk{k}:= \cfr / \mathbf{m}^{k+1}$ be the Artinian ring for each $k \in \mathbb{N}$, and $\hat{\cfr}:= \varprojlim_{k} \cfrk{k}$ be the completion of $\cfr$. We further equip $\cfr$, $\cfrk{k}$ and $\hat{\cfr}$ with the natural monoid homomorphism $Q \rightarrow \cfr$, $m \mapsto q^m$, giving them the structure of a {\em log ring} (see \cite[Definition 2.11]{gross2011real}); the corresponding log spaces will be denoted as $\logs$, $\logsk{k}$ and $\logsf$ respectively. In particular, we will call $\logsk{0}$ the standard $Q$-log point.
\end{notation}

Throughout this paper, we are often dealing with two \v{C}ech covers $\mathcal{V} = (V_\alpha)_{\alpha}$ and $\mathcal{U} = (U_i)_{i \in \mathbb{N}}$ at the same time and also $k$-th order thickenings, so we will adapt the following (rather unusual) notations from \cite{chan2019geometry}: The top left hand corner in a notation $\prescript{k}{}{\spadesuit}$ refers to the order of $\spadesuit$, while the bottom right hand corner is reserved for the \v{C}ech indices. We also write $\spadesuit_{\alpha_0 \cdots \alpha_{\ell}}$ for the \v{C}ech indices of $\mathcal{V}$ and $\spadesuit_{i_0 \cdots i_l}$ for the \v{C}ech indices of $\mathcal{U}$, and if they appear at the same time, we write $\spadesuit_{\alpha_0 \cdots \alpha_{\ell},i_0 \cdots i_l}$. 
\section{Abstract algebra for deformations of pairs}\label{sec:abstract_algebra_for_deformation_of_pairs}

Here we review the abstract algebra needed for the deformation theory of pairs from \cite{iacono2018deformations}.
First recall that a {\em differential graded Lie algebra} (or {\em dgLa}) is a triple 
$$(L^* , d , [\cdot,\cdot]),$$
where $L^* = \bigoplus_{i \in \inte} L^i$ is a graded vector space, $[\cdot,\cdot] : L^{*} \otimes L^* \rightarrow L^*$ is a graded skew-symmetric pairing satisfying the Jacobi identity $[a,[b,c]] + (-1)^{|a||b|+|a||c|} [b,[c,a]] + (-1)^{|a||c| + |b| |c|}[ c,[a,b]] = 0$ for homogeneous elements $a,b,c \in L^*$, and $d: L^{*} \rightarrow L^{*+1}$ is a degree $1$ differential satisfying $d^2 =0$ and the Leibniz rule $d[a,b] = [da,b] + (-1)^{|a|} [a,db]$ for homogeneous elements $a, b \in L^*$; here $|a|$ denotes the degree of a homogeneous element $a$.

Let $\mathcal{O}$ be a unitary $\cfr$-algebra and $(M^*,d_M)$ be a bounded cochain complex of free $\mathcal{O}$-modules of finite rank. Also let $\der{R}{\mathcal{O}}{}$ be the Lie algebra of $\cfr$-linear derivations on $\mathcal{O}$, and consider an $\cfr$-linear Lie algebra homomorphism $\iota:\sla{} \rightarrow \der{\cfr}{\mathcal{O}}{}$ for some Lie algebra $(\sla{},[\cdot,\cdot])$ over $\cfr$ (i.e. $\sla{}$ is an $\cfr$-module and $[\cdot,\cdot]$ is $\cfr$-linear). To simplify notations, we will write $h(r)$ instead of $\iota(h)(r)$ for $h \in \sla{}$ and $r \in \mathcal{O}$. The following definition is a slight modification of the one from \cite[p.1219, 2nd paragraph after proof of Corollary 3.3]{iacono2018deformations}:

\begin{definition}\label{def:Lie_algebra_of_pairs}
	We treat $\sla{}$ as a dgLa concentrated at degree $0$, and let
	$$
	\sla{}^*(M^*) := \{ (h,u) \in \sla{} \times \hom^*_{\cfr}(M^*,M^*) \mid u(rx) = h(r) x + r u(x) \ \text{for $r\in \mathcal{O}, \ x \in M^*$}  \},
	$$
	where the grading on $\sla{}^*(M^*)$ is inherited from $\hom^*_{\cfr}(M^*,M^*)$. We equip $\sla{}^*(M^*)$ with a dgLa structure by the formulas 
	$$d(h,u) := (0, du),\quad [(h_1,u_1),(h_2,u_2)]:= ([h_1,h_2],[u_1,u_2]),$$
	where $d u := d_M \circ u + (-1)^{|u|} u \circ d_M$.
	The natural projection from $\sla{}^*(M^*)$ to $\sla{}$ gives a surjective morphism $\anchor{} : \sla{}^*(M^*) \rightarrow \sla{}$ of dgLa's called the {\em anchor map}. 	
\end{definition}

From its definition, we have $\sla{}^i(M) = \hom^i_{\mathcal{O}}(M^*,M^*)$ if $i\neq 0$, and the exact sequence of dgLa's 
\begin{equation}\label{eqn:abstract_algebra_exact_sequence_of_deformation_of_pairs}
 0 \rightarrow \hom^*_{\mathcal{O}}(M^*,M^*) \rightarrow \sla{}^*(M^*) \rightarrow \sla{} \rightarrow 0.
\end{equation}

\begin{definition}[cf. Definition 2.9 in \cite{iacono2018deformations}]\label{def:algebra_automorphism_of_pair}
	With $(\mathcal{O},M^*)$ as above, let $\text{Aut}_{\cfr}(\mathcal{O},M^*)$ be its automorphism group which consists of pairs $(\Theta,b)$, where $\Theta: \mathcal{O} \rightarrow \mathcal{O}$ is an automorphism of $\cfr$-algebras and $b : M^* \rightarrow M^*$ is an automorphism of complexes of $\cfr$-modules, such that $b(rm) = \Theta(r) b(m)$ for $r \in \mathcal{O}$ and $m \in M^*$.   
\end{definition}

Following the discussion in \cite[\S 2]{iacono2018deformations}, given any nilpotent element $(h,u)$ in $\sla{}^0(M)$, its exponential $(\exp(h), \exp(u))$ (where we abuse notations and simply write $h$ for $\iota(h)$) gives an element in $\text{Aut}_{\cfr}(\mathcal{O},M^*)$. 

If $M$ is a free $\mathcal{O}$-module of rank $m$, we write $\det M := \bigwedge^m_{\mathcal{O}} M$. If $M^*$ is a bounded complex of finite rank free $\mathcal{O}$-modules, we set $\det M^* := \bigotimes_{\mathcal{O}}^{l=\text{ev}} \det M^l \otimes_{\mathcal{O}} \bigotimes_{\mathcal{O}}^{l=\text{odd}} \det (M^l)^{\vee}$ (here $^{\vee}$ refers to the dual $\mathcal{O}$-module). For any $u \in \hom^0_{\cfr}(M^*,M^*)$, we let $\trace(u) \in \hom_{\cfr}(\det M^*,\det M^*)$ be the natural induced map as defined in \cite[Definitions 2.7 and 2.8]{iacono2018deformations}. This gives a natural map
$$\trace: \sla{} \times \hom^*_{\cfr}(M^*,M^*) \rightarrow \sla{} \times \hom_{\cfr}(\det M^*,\det M^*),\quad (h,u) \mapsto (h,\trace(u)),$$
which induces a morphism $\trace: \sla{}^*(M^*) \rightarrow \sla{}^0(\det M^*)$ of dgLa's (where $\det M^*$ is treated as a complex concentrated in degree $0$). 

\section{Gluing construction of an almost dgla for smoothing pairs}\label{sec:gluing_construction}

In this section, we extend and apply the gluing construction developed in \cite{chan2019geometry}. Note that we will work with almost dgLa's here instead of the almost dgBV algebras in \cite{chan2019geometry}.

\subsection{Abstract local deformation and patching data}\label{sec:abstract_deformation_data}

Let $(X,\mathcal{O}_X)$ be a $d$-dimensional compact complex analytic space. We fix an open cover $\mathcal{V} = \{V_\alpha\}_\alpha$ of $X$ which consists of Stein open subsets $V_{\alpha} \subset X$. 
In geometric situations, there will be a local smoothing model $\prescript{k}{}{\mathbf{V}}_{\alpha}$ over $\cfrk{k}$ of each $V_{\alpha} \subset X$. 
To patch these local models together, we will need another Stein open cover $\mathcal{U}$ on $X$:
\begin{notation}\label{not:open_stein_covers}
	Fix, once and for all, a cover $\mathcal{U}$ of $X$ which consists of a countable collection of Stein open subsets $\mathcal{U} = \{U_i\}_{i \in \mathbb{N}}$ forming a basis of topology. We refer readers to \cite[Chapter IX Theorem 2.13]{demailly1997complex} for the existence of such a cover. Note that an arbitrary finite intersection of Stein open subsets remains Stein. 
\end{notation}

\begin{definition}\label{def:abstract_deformation_data}
	An {\em abstract deformation datum} $\bva{} = (\bva{0}^*, \{\bva{k}^*_{\alpha}\}_{k,\alpha}, \{\rest{k,l}_{\alpha}\}_{k\geq l,\alpha})$ consists of 
	\begin{itemize}
		\item a sheaf $\bva{0}^*$ of dgLa's with bounded degrees;
		\item for each $k \in \mathbb{Z}_{\geq 0}$ and $\alpha$, a sheaf $\bva{k}^*_{\alpha}$ of dgLa's with bounded degrees on $V_{\alpha}$ equipped with the structure of sheaf of $\cfrk{k}$-modules such that the structures $[\cdot,\cdot]$ and $d$ are $\cfrk{k}$-linear, and 
		\item for $k\geq l$ and each $\alpha$, a surjective degree-preserving morphism $\rest{k,l}_{\alpha} : \bva{k}^*_{\alpha} \rightarrow \bva{l}^*_{\alpha}$ of dgLa's which induces an isomorphism upon tensoring $\bva{k}^*_{\alpha}$ with $\cfrk{l}$
	\end{itemize}
	such that
	\begin{enumerate}
		\item $\bva{0}_{\alpha}^* = \bva{0}^*|_{V_{\alpha}}$;
		\item $\bva{k}^*_{\alpha}$ is flat over $\cfrk{k}$, i.e. the stalk $(\bva{k}_{\alpha}^*)_x$ is flat over $\cfrk{k}$ for all $x \in V_{\alpha}$;
		\item the adjoint homomorphism $\text{ad} : \bva{k}_{\alpha}^* \rightarrow \der{(\cfrk{k})}{\bva{k}_{\alpha}^*}{*}$ (where $\der{(\cfrk{k})}{}{*}$ denotes the graded space of $\cfrk{k}$-linear derivations) is injective; and
		\item $R\Gamma^j(V_{\alpha_0 \cdots \alpha_{\ell}},\bva{k}_{\alpha}) = 0$ and $R\Gamma^j(U_{i_0\cdots i_l},\bva{k}_{\alpha}) = 0$ for all $j>0$, $V_{\alpha_0 \cdots \alpha_{\ell}}:= V_{\alpha_0} \cap \cdots \cap V_{\alpha_{\ell}}$ and $U_{i_0 \cdots i_l} := U_{i_0} \cap \cdots \cap U_{i_l}$ with $U_{i_0}, \dots, U_{i_l} \subset V_\alpha$.
	\end{enumerate}
\end{definition}

Condition $(4)$ in Definition \ref{def:abstract_deformation_data} is a very mild assumption. It holds when, e.g., $\bva{k}_{\alpha}$ is a tensor product of coherent sheaves with a (possibly infinite dimensional) vector space. 

For smoothing of a singular variety $X$ (see \S \ref{sec:data_from_gross_siebert}), we take the sheaf $\sla{k}_{\alpha}$ of of relative log derivations over $\cfrk{k}$, which is a sheaf of Lie algebras that controls the local deformations of the log space $\prescript{k}{}{\mathbf{V}}_{\alpha}$.

When we consider smoothing of a pair $(X, \mathfrak{C}^*)$ (see \S \ref{sec:pair_data_from_gross_siebert}), we take the sheaf $\gla{k}_\alpha$ of dgLa's, obtained from $\sla{k}_{\alpha}$ using the algebraic construction in Definition \ref{def:Lie_algebra_of_pairs}, which controls local deformations of pairs.
Here is a caveat: For the case of pairs, condition $(3)$ in Definition \ref{def:abstract_deformation_data} does {\em not} hold, and we will need to suitably modify this definition as described in \S \ref{sec:weaken_assumption}. 

\begin{notation} Given two elements $\mathfrak{a} \in \bva{k_1}_{\alpha}^*$, $\mathfrak{b} \in \bva{k_2}_{\alpha}^*$ and $l \leq \text{min}\{k_1,k_2\}$, we say that $\mathfrak{a} = \mathfrak{b} \ \text{(mod $\mathbf{m}^{l+1}$)}$ if and only if $\rest{k_1,l}_{\alpha}(\mathfrak{a}) = \rest{k_2,l}_{\alpha}(\mathfrak{b})$.
\end{notation}

\begin{definition}\label{def:patching_datum}
	Given an abstract deformation datum $\bva{} = (\bva{0}^*, \{\bva{k}^*_{\alpha}\}_{k,\alpha}, \{\rest{k,l}_{\alpha}\}_{k\geq l,\alpha})$, 
	a {\em patching datum} $\patch{} = \{\patch{k}_{\alpha\beta,i}\}$ (with respect to $\mathcal{U}, \mathcal{V}$) consists of, for each $k \in \mathbb{N}$ and $(U_i; V_\alpha, V_\beta)$ with $U_i \subset V_{\alpha\beta} := V_{\alpha} \cap V_{\beta}$,
	a sheaf isomorphism $\patch{k}_{\alpha\beta,i}:\bva{k}_{\alpha}^*|_{U_i} \to \bva{k}_{\beta}^*|_{U_i}$ over $\cfrk{k}$ preserving $[\cdot,\cdot]$ and $d$ and fitting into the diagram
	$$
	\xymatrix@1{ \bva{k}_{\alpha}^*|_{U_i} \ar[r]^{\patch{k}_{\alpha\beta,i}} \ar[d]^{\rest{k,0}_{\alpha}} &   \bva{k}_{\beta}^*|_{U_i} \ar[d]^{\rest{k,0}_{\beta}}\\
		\bva{0}^*|_{U_i} \ar@{=}[r] &  \bva{0}^*|_{U_i}, 
	}
	$$
	such that: 
	\begin{enumerate}
		\item $\patch{k}_{\beta\alpha,i} = \patch{k}_{\alpha \beta,i}^{-1}$, $\patch{0}_{\alpha \beta,i} \equiv \text{id}$;
		\item for $k>l$ and $U_i \subset V_{\alpha\beta}$, there exists $\resta{k,l}_{\alpha\beta,i} \in \bva{l}_{\alpha}^{0}(U_i)$  with $\resta{k,l}_{\alpha\beta,i} = 0 \ \text{(mod $\mathbf{m}$)}$ such that 
		\begin{equation}\label{eqn:bva_different_order_comparison}
		\patch{l}_{\beta\alpha,i} \circ \rest{k,l}_{\beta} \circ \patch{k}_{\alpha\beta,i} = \exp\left([\resta{k,l}_{\alpha\beta,i},\cdot]\right) \circ \rest{k,l}_{\alpha};
		\end{equation}
		
		\item for $k\in \mathbb{N}$ and $U_i, U_j \subset V_{\alpha \beta}$, there exists $\patchij{k}_{\alpha\beta,ij} \in \bva{k}_{\alpha}^{0}(U_i\cap U_j)$ with $\patchij{k}_{\alpha\beta,ij} = 0 \ \text{(mod $\mathbf{m}$)}$ such that
		\begin{equation}\label{eqn:higher_order_U_i_U_j_different}
		\left(\patch{k}_{ \beta \alpha,j}|_{U_i \cap U_j}\right)\circ \left(\patch{k}_{\alpha \beta,i}|_{U_i \cap U_j}\right) = \exp\left([\patchij{k}_{\alpha\beta,ij}, \cdot]\right); and
		\end{equation}
		
		\item for $k\in \mathbb{N}$ and $U_i \subset V_{\alpha\beta\gamma}:=V_\alpha \cap V_\beta \cap V_\gamma$, there exists $\cocyobs{k}_{\alpha\beta\gamma,i} \in \bva{k}_{\alpha}^{0}(U_i)$  with $\cocyobs{k}_{\alpha\beta\gamma,i} = 0 \ \text{(mod $\mathbf{m}$)}$ such that 
		\begin{equation}\label{eqn:higher_order_cocycle_different}
		\left(\patch{k}_{\gamma \alpha,i}|_{U_i}\right) \circ \left(\patch{k}_{\beta \gamma,i}|_{U_i}\right) \circ \left(\patch{k}_{\alpha \beta ,i}|_{U_i}\right) = \exp\left([\cocyobs{k}_{\alpha\beta\gamma,i},\cdot]\right).
		\end{equation}
	\end{enumerate}
\end{definition}

In geometric situations such as smoothing of the variety $X$, the patching isomorphism $\patch{k}_{\alpha\beta, i}$ (note that this is actually denoted as $\patchh{k}_{\alpha\beta,i}$ in \S \ref{sec:data_from_gross_siebert}) is induced from the local uniqueness of the local smoothing model $\prescript{k}{}{\mathbf{V}}_{\alpha}$. Equations \eqref{eqn:bva_different_order_comparison}, \eqref{eqn:higher_order_U_i_U_j_different} and \eqref{eqn:higher_order_cocycle_different} say that local automorphisms of the local models are exponentiation of the Lie bracket with local vector fields $\resta{k,l}_{\alpha\beta,i}$'s, $\patchij{k}_{\alpha\beta,ij}$'s and $\cocyobs{k}_{\alpha\beta\gamma,i}$'s. The key point is that we do {\em not} require the patching isomorphisms $\patch{k}_{\alpha\beta,i}$'s to be compatible directly but rather the discrepancies are captured by the Lie bracket with local sections of the sheaves $\bva{k}_{\alpha}^*$'s.

\begin{definition}\label{def:morphism_of_deformation_data}
	A {\em morphism} $\morph{} = \{\morph{k}_{\alpha}\}_{k,\alpha} :  (\gla{}=\{\gla{k}^*_{\alpha}\}_{k,\alpha},\patchp{}= \{\patchp{k}_{\alpha\beta,i}\}) \rightarrow (\bva{}=\{\bva{k}^*_{\alpha}\}_{k,\alpha},\patch{}=\{\patch{k}_{\alpha\beta,i}\})$ consists of $\cfrk{k}$-linear morphisms $\morph{k}_{\alpha} : \gla{k}^*_{\alpha} \rightarrow \bva{k}^*_{\alpha}$ between sheaves of dgLa's over $V_{\alpha}$ satisfying the following conditions:
	\begin{enumerate}
		\item for $k\geq l$ and each $\alpha$, we have $ \rest{k,l}_{\alpha} \circ \morph{k}_{\alpha} = \morph{l}_{\alpha} \circ \restp{k,l}_{\alpha}$;
		\item for each $k$ and $U_i \subset V_{\alpha\beta}$, we have $\patch{k}_{\alpha\beta,i} \circ \morph{k}_{\alpha} = \morph{k}_{\beta} \circ \patchp{k}_{\alpha\beta,i}$;
		\item for $k>l$ and $U_i \subset V_{\alpha\beta}$, we have $\morph{l}_{\alpha}(\restap{k,l}_{\alpha\beta,i}) = \resta{k,l}_{\alpha\beta,i}$; 
		\item for each $k$ and $U_i,U_j \subset V_{\alpha\beta}$, we have $\morph{k}_{\alpha} (\patchijp{k}_{\alpha\beta,ij}) = \patchij{k}_{\alpha\beta,ij}$; and
		\item for each $k$ and $U_i \subset V_{\alpha\beta\gamma}$, we have $\morph{k}_{\alpha} (\cocyobsp{k}_{\alpha\beta\gamma,i}) = \cocyobs{k}_{\alpha\beta\gamma,i}$;
	\end{enumerate}
	here we have used $\restp{k,l}_{\alpha}$'s,  $\patchp{k}_{\alpha\beta,i}$'s, $\restap{k,l}_{\alpha\beta,i}$'s, $\patchijp{k}_{\alpha\beta,ij}$'s and $\cocyobsp{k}_{\alpha\beta\gamma,i}$'s to denote the deformation and patching data associated to $\gla{}$.
\end{definition}

In this paper, we will only be interested in the case that the morphism $\{\morph{k}_{\alpha}\}_{k,\alpha}$ is coming from either the trace map $\trace: \sla{}^*(M^*) \rightarrow \sla{}^0(\det M^*)$ or the anchor map $\alpha : \sla{}^*(M^*)  \rightarrow \sla{}^*$ in \S \ref{sec:abstract_algebra_for_deformation_of_pairs}. 

\begin{remark}\label{rem:surjective_morphism}
If we assume the injectivity of the adjoint map $\text{ad} : \bva{k}^*_{\alpha} \rightarrow \der{(\cfrk{k})}{\bva{k}^*_{\alpha}}{*}$ and surjectivity of the morphisms $\morph{k}_{\alpha}$'s, then conditions $(3)-(5)$ in Definition \ref{def:morphism_of_deformation_data} follow from conditions $(1)-(2)$ because the elements $\resta{k,l}_{\alpha\beta,i}$'s, $\patchij{k}_{\alpha\beta,ij}$'s and $\cocyobs{k}_{\alpha\beta\gamma,i}$'s are determined by the maps $\rest{k,l}_{\alpha}$'s and $\patch{k}_{\alpha\beta,i}$'s by the equations in Definition \ref{def:patching_datum} and the injectivity of $\text{ad}$.
\end{remark}

\subsection{Construction of the almost dgLa}\label{sec:construction_of_dgLa}

\subsubsection{Simplicial sets}\label{sec:simplicial_sets}
To construct the almost dgLa using simplicial methods, we first recall some standard definitions and facts on the simplicial sets $\mathcal{A}^*(\simplex_{\bullet})$ of polynomial differential forms with coefficients in $\comp$, following the notations from \cite[\S 3.1]{chan2019geometry}.

\begin{notation}\label{not:simplicial_set}
	We let $\text{Mon}$ (resp. $\text{sMon}$) be the category of finite ordinals $[n] = \{0,1,\dots,n\}$ in which morphisms are increasing maps (resp. strictly increasing maps).
	We denote by $\mathtt{d}_{i,n} : [n-1] \rightarrow [n]$ the unique strictly increasing map which skips the $i$-th element, and by $\mathtt{e}_{i,n} : [n+1] \rightarrow [n]$ be given by $\mathtt{e}_{i,n}(j) = j$ if $j \leq i$ and $\mathtt{e}_{i,n}(j) = j -1$ if $j>i$. 
\end{notation}

Note that every morphism in $\text{Mon}$ can be decomposed as a composition of the maps $\mathtt{d}_{i,n}$'s and $\mathtt{e}_{i,n}$'s, and any morphism in $\text{sMon}$ can be decomposed as a composition of the maps $\mathtt{d}_{i,n}$'s. 

\begin{definition}[\cite{weibel1995introduction}]\label{def:simplicial_and_cosimplicial_set}
	Let $\mathtt{C}$ be a category. A {\em (semi-)simplicial object in $\mathtt{C}$} is a contravariant functor $\mathtt{A}(\bullet) : \text{Mon} \rightarrow  \mathtt{C}$ (resp. $\mathtt{A}(\bullet) : \text{sMon} \rightarrow  \mathtt{C}$), and a {\em (semi-)cosimplicial object in $\mathtt{C}$} is a covariant function $\mathtt{A}(\bullet) : \text{Mon} \rightarrow \mathtt{C}$ (resp. $\mathtt{A}(\bullet) : \text{sMon} \rightarrow \mathtt{C}$).
\end{definition}

\begin{definition}[\cite{griffiths1981rational}]\label{def:simplicial_de_rham}
	Consider the standard $n$-simplex $\simplex_n := \{(x_0, \dots, x_n) \in \mathbb{R}^{n+1} \mid \sum_{i=0}^n x_i = 1\}$. The {\em space of polynomial differential forms with coefficients in $\comp$ on $\simplex_n$} is defined as the differential graded algebra (abbreviated as dga)
	$$
	\mathcal{A}^*(\simplex_n) := \frac{\text{Sym}^*\left(\comp \langle x_0,\dots,x_n,dx_0,\dots,dx_n \rangle \right)}{\left(\sum_{i=0}^n x_i -1, \sum_{i=0}^n dx_i\right)},
	$$
	where $\comp \langle x_0,\dots,x_n,dx_0,\dots,dx_n \rangle $ is the graded vector space generated by $x_i$'s and $dx_i$'s with $\deg(x_i) = 0$, $\deg(dx_i) =1$, and the degree $1$ differential $d$ is defined by $d(x_i) = dx_i$ and the Leibniz rule.
	
	Given $a : [n] \rightarrow [m]$ in $\text{Mon}$, we let $a^*:=\mathcal{A}(a) : \mathcal{A}^*(\simplex_m) \rightarrow \mathcal{A}^*(\simplex_n)$ be the unique dga morphism satisfying $a^*(x_j) = \sum_{i \in [n]: a(i) = j} x_i$ and $a^*(x_{j}) = 0$ if $j \neq a(i)$ for any $i \in [n]$. From this we obtain a simplicial object in the category of dga's, which we denote by $\mathcal{A}^*(\simplex_\bullet)$.
\end{definition}

\begin{notation}\label{not:de_rham_form_on_simplex_boundary}
	We denote by $\simplexbdy_n$ the boundary of $\simplex_n$, and let 
	\begin{equation}
	\mathcal{A}^*(\simplexbdy_n) : =\{(\alpha_0,\dots,\alpha_n) \mid \alpha_{i}  \in \mathcal{A}^*(\simplex_{n-1}), \ \mathtt{d}_{i,n-1}^* (\alpha_j) =\mathtt{d}_{j-1,n-1}^*(\alpha_i) \ \text{for $0 \leq i<j \leq n$} \}
	\end{equation}
	be the space of polynomial differential forms on $\simplexbdy_n$. There is a natural restriction map defined by $\beta|_{ \simplexbdy_n} := (\mathtt{d}_{0,n}^*(\beta),\dots,\mathtt{d}^*_{n,n}(\beta))$ for $\beta \in \mathcal{A}^*(\simplex_n)$. 
	Similarly, we let $\Lambda^k_n \subset \simplexbdy_n$ be the $k$-th horn, and
	\begin{equation}\label{eqn:differential_forms_on_k_horn}
	\mathcal{A}^*(\Lambda^k_n) : =\{(\alpha_0,\dots,\alpha_{k-1},\alpha_{k+1},\dots, \alpha_n) \mid \alpha_{i}  \in \mathcal{A}^*(\simplex_{n-1}), \ \mathtt{d}_{i,n-1}^* (\alpha_j) =\mathtt{d}_{j-1,n-1}^*(\alpha_i) \ \text{for $0 \leq i<j \leq n$} \}
	\end{equation}
	be the space of polynomial differential forms on $\Lambda^k_n$, with a natural restriction map $\beta|_{\Lambda^k_n}$ defined in a similar way.
\end{notation}

The following extension lemma will be frequently used in subsequent constructions:

\begin{lemma}[Lemma 9.4 in \cite{griffiths1981rational}]\label{lem:simplicial_de_rham_extension}
	For any $\vec{\alpha} = (\alpha_0,\dots,\alpha_n) \in \mathcal{A}^*(\simplexbdy_n)$, there exists $\beta \in \mathcal{A}^*(\simplex_n)$ such that $\beta|_{\simplexbdy_n} = \vec{\alpha}$. 
\end{lemma}

\begin{notation}\label{not:homotopy_simplex_notation}
	We let $\hsimplex_{m,n} := \simplex_m \times \simplex_n$, and 
	\begin{equation}\label{eqn:homotopy_simplex_de_rham}
	\mathcal{A}^*(\hsimplex_{m,n}) := \mathcal{A}^*(\simplex_m) \otimes_{\comp} \mathcal{A}^*(\simplex_n).
	\end{equation}
	There are two sets of restriction maps: $\mathtt{d}_{j,m}^* : \mathcal{A}^*(\hsimplex_{m,n}) \rightarrow \mathcal{A}^*(\hsimplex_{m-1,n})$ induced from that on $\simplex_m$, and $\mathtt{d}_{j,n}^* : \mathcal{A}^*(\hsimplex_{m,n})  \rightarrow \mathcal{A}^*(\hsimplex_{m,n-1})$ induced from that on $\simplex_n$.  
\end{notation}

\begin{notation}\label{not:de_rham_form_on_homotopy_simplex_boundary}
	We denote by $\hsimplexbdy_n$ the boundary of $\hsimplex_n$, and let
	\begin{equation*}
	\mathcal{A}^*(\hsimplexbdy_{m,n}) := \left\{\displaystyle{\substack{(\alpha_0,\dots,\alpha_m, \beta_0 ,\dots ,\beta_n)\\\alpha_{i} \in \mathcal{A}^*(\hsimplex_{m-1,n}),\ \beta_{i} \in \mathcal{A}^*(\hsimplex_{m,n-1}) }} \ \Bigg|\ \displaystyle{\substack{ 	\mathtt{d}_{i,m-1}^* (\alpha_j) = \mathtt{d}_{j-1,m-1}^*(\alpha_i) \text{ for $0 \leq i<j \leq m$} \\
			\mathtt{d}_{i,n-1}^* (\beta_j) = \mathtt{d}_{j-1,n-1}^*(\beta_i) \text{ for $0 \leq i<j \leq n$} \\
			\mathtt{d}_{i,n}^* (\alpha_j) = \mathtt{d}_{j,m}^*(\beta_i) \text{ for $0\leq i\leq n$ and $0\leq j \leq m$}}}\right\}
	\end{equation*}
	be the space of polynomial differential forms on $\hsimplexbdy_n$. There is a natural restriction map defined by $\gamma|_{ \hsimplexbdy_{m,n}} := (\mathtt{d}_{0,m}^*(\gamma),\dots,\mathtt{d}^*_{m,m}(\gamma), \mathtt{d}_{0,n}^*(\gamma),\dots,\mathtt{d}_{n,n}^*(\gamma))$ for $\gamma \in \mathcal{A}^*(\hsimplex_{m,n})$. 
\end{notation}
\begin{lemma}\label{lem:homotopy_simplicial_de_rham_extension}
	For any $(\alpha_0,\dots,\alpha_m, \beta_0,\dots,\beta_n) \in \mathcal{A}^*(\hsimplexbdy_{m,n})$, there exists $\gamma \in \mathcal{A}^*(\hsimplex_{m,n})$ such that $\gamma|_{\hsimplexbdy_{m,n}} = (\alpha_0,\dots,\alpha_m, \beta_0,\dots,\beta_n)$. 
\end{lemma}
This variation of Lemma \ref{lem:simplicial_de_rham_extension} can be proven by the same technique as in \cite[Lemma 9.4]{griffiths1981rational}.

\subsubsection{Gluing morphisms}\label{sec:gluing_morphism}

On $V = V_{\alpha_0 \cdots \alpha_{\ell}}$, we consider the covering $\mathcal{U}  = \mathcal{U}_{\alpha_0 \cdots \alpha_{\ell}} := \{ U_{i} \mid U_{i} \subset V_{\alpha_0 \cdots \alpha_{\ell}}$ parametrized by the mult-index set $\mathcal{I} := \{ (i_0,\dots,i_l) \mid U_{i_j} \in \mathcal{U}_{\alpha_0 \cdots \alpha_{\ell}} \}$. Assuming that $\mathcal{U}$ is an acyclic cover for a sheaf $\bva{}^p$ on $V$ for each $p$, we have the following definition. 

\begin{definition}[see e.g. \cite{whitney2012geometric, dupont1976simplicial, fiorenza2012differential}]\label{def:thom_whitney_general}
	The {\em Thom-Whitney complex} is defined as $\twc{}^{*,*}(\bva{}) := \bigoplus_{p,q} \twc{}^{p,q}(\bva{})$, where
	\begin{equation*}
	\twc{}^{p,q}(\bva{}) := \left\{ (\varphi_{i_0 \cdots i_l})_{(i_0,\dots,  i_l) \in \mathcal{I}} \ \big|\ \varphi_{i_0 \cdots i_l} \in \mathcal{A}^q(\simplex_l) \otimes_\comp \bva{}^p(U_{i_0\cdots i_l}), \
	\mathtt{d}_{j,l}^* (\varphi_{i_0 \cdots i_l}) = \varphi_{i_0 \cdots \hat{i}_j \cdots i_l} |_{U_{i_0 \cdots i_l}} \right\}.
	\end{equation*}
	It is a dgLa with the Lie bracket $[\cdot,\cdot]$ and differential $\pdb+d$ defined component-wise by
	\begin{align*}
	[\alpha_I \otimes v_I , \beta_I \otimes w_I ] := (-1)^{|v_I| |\beta_I|} (\alpha_I \wedge \beta_I) \otimes [v_I,w_I], \\
	\pdb (\alpha_I \otimes v_I) := (d\alpha_I) \otimes v_I ,\quad d(\alpha_I \otimes v_I) = (-1)^{|\alpha_I|} \alpha_I  \otimes (d v_I)
	\end{align*}
	for $\alpha_I, \beta_I \in \mathcal{A}^*(\simplex_l)$ and $v_I, w_I \in \bva{}^*(U_I),$ where $l = |I| - 1$.
\end{definition}

The complex $\twc{}^{p,*}(\bva{})$ (resp. the total complex $\twc{}^*(\bva{})$) is quasi-isomorphic to the \v{C}ech complex $\check{\mathcal{C}}^*(\mathcal{U},\bva{}^p)$ (resp. the total \v{C}ech complex $\check{\mathcal{C}}^*(\mathcal{U},\bva{}^*)$). 

\begin{remark}
	We use the notation $\pdb$ since it plays the role of the Dolbeault operator in the classical deformation theory of complex manifolds.
\end{remark}

\begin{notation}\label{not:local_thom_whitney_complex}
Equipping $\bva{}^* := \mathcal{A}^*(\simplex_n) \otimes \bva{k}^*_{\alpha_i}|_{V_{\alpha_0 \cdots \alpha_{\ell}}}$ with the natural dgLa structure $([\cdot,\cdot],\prescript{k}{}{d}_{\alpha_i,\simplex_n})$, 
we let $\twc{k}^{*,*}_{\alpha_i;\alpha_0 \cdots \alpha_{\ell}}(\simplex_n) := \twc{}^{*,*}(\bva{})$, which is equipped with the differential $\prescript{k}{}{\pdb}_{\alpha_i} + \prescript{k}{}{d}_{\alpha_i,\simplex_n}$ and the Lie bracket $[\cdot,\cdot]$. It is naturally equipped with the face map 
$$\mathtt{d}_{j,n}^* : \twc{k}^{*,*}_{\alpha_i;\alpha_0 \cdots \alpha_{\ell}}(\simplex_n) \rightarrow \twc{k}^{*,*}_{\alpha_i;\alpha_0 \cdots \alpha_{\ell}}(\simplex_{n-1})$$
and the restriction map $\restmap_{\alpha_j}$ defined component-wise by
	$$
	\restmap_{\alpha_j} \left( (\varphi_{I})_{I \in \mathcal{I}} \right) = (\varphi_I)_{I \in \mathcal{I}'}
	$$
	for $(\varphi_I)_{I\in \mathcal{I}} \in  \twc{k}^{*,*}_{\alpha_i; \alpha_0 \cdots \widehat{\alpha_j} \cdots \alpha_\ell} (\simplex_n)$, where $\mathcal{I}' = \{ (i_0,\dots,i_l) \in \mathcal{I} \mid U_{i_j} \subset V_{\alpha_0\cdots \alpha_{\ell}} \}$.
\end{notation}

For each $k \in \mathbb{N}$ and any pair $V_\alpha, V_\beta \in \mathcal{V}$, a {\em gluing isomorphism} 
\begin{equation}\label{eqn:g_alpha_beta_map}
\glue{k}_{\alpha \beta }(\simplex_n) : \twc{k}^{*,*}_{\alpha;\alpha\beta }(\simplex_n) \rightarrow \twc{k}^{*,*}_{\beta ; \alpha \beta}(\simplex_n)
\end{equation}
is a collection of maps $(\glue{k}_{\alpha\beta, I}(\simplex_n))_{I\in \mathcal{I}}$ such that for $\varphi =(\varphi_I)_{I\in \mathcal{I}} \in \twc{k}^{*,*}_{\alpha;\alpha \beta}(\simplex_n)$ with $\varphi_I \in \mathcal{A}^*(\simplex_l) \otimes \mathcal{A}^*(\simplex_n)\otimes \bva{k}^*_\alpha(U_I)$, we have $\left( \glue{k}_{\alpha\beta}(\simplex_n)(\varphi) \right)_I = \glue{k}_{\alpha\beta, I}(\simplex_n)(\varphi_I)$. It is required to preserve the algebraic structures and satisfy the following condition:
\begin{condition}\label{assum:induction_hypothesis}
	\begin{enumerate}
		\item for $U_i \subset V_\alpha \cap V_\beta$, we have 
		\begin{equation}\label{eqn:alpha_beta_gluing_explicit_form_at_point}
		\glue{k}_{\alpha\beta, i}(\simplex_n) = \exp([ \iauto{k}_{\alpha\beta,i}(\simplex_n), \cdot ]) \circ \patch{k}_{\alpha\beta,i}
		\end{equation}
		for some element $\iauto{k}_{\alpha\beta,i}(\simplex_n) \in \mathcal{A}^0(\simplex_n) \otimes \bva{k}^{0}_{\beta}(U_i)$ with $\iauto{k}_{\alpha\beta,i}(\simplex_n) = 0 \ \text{(mod $\mathbf{m}$)}$;
		
		\item for $U_{i_0}, \dots, U_{i_l} \subset V_\alpha \cap V_{\beta}$, we have 
		\begin{equation}\label{eqn:alpha_beta_gluing_explicit_form_on_simplex}
		\glue{k}_{\alpha\beta,i_0\cdots i_l}(\simplex_n) = \exp([\sauto{k}_{\alpha\beta, i_0 \cdots i_l}(\simplex_n),\cdot]) \circ \left(\glue{k}_{\alpha\beta,i_0}(\simplex_n)|_{U_{i_0 \cdots i_l}}\right),
		\end{equation}
		for some element $\sauto{k}_{\alpha\beta, i_0 \cdots i_l} \in \mathcal{A}^0(\simplex_l) \otimes \mathcal{A}^0(\simplex_n)\otimes  \bva{k}^{0}_\beta(U_{i_0 \cdots i_l})$ with $\sauto{k}_{\alpha\beta, i_0 \cdots i_l}(\simplex_n) = 0 \ \text{(mod $\mathbf{m}$)}$;
		
		\item the elements $\sauto{k}_{\alpha\beta, i_0 \cdots i_l}$'s satisfy the relation: 
		\begin{equation}\label{eqn:alpha_beta_gluing_explicit_form_on_simplex_boundary_relation}
		\mathtt{d}_{j,l}^*( \sauto{k}_{\alpha\beta, i_0 \cdots i_l}(\simplex_n)) =
		\begin{dcases}
		\sauto{k}_{\alpha\beta, i_0 \cdots \widehat{i_j} \cdots i_l}(\simplex_n) & \text{for $j>0$},\\
		\sauto{k}_{\alpha\beta, \widehat{i_0}  \cdots i_l}(\simplex_n) \bchprod \autoij{k}_{\alpha\beta, i_0 i_1}(\simplex_n)  & \text{for $j=0$}, 
		\end{dcases}
		\end{equation}
		where $\bchprod$ refers to the Baker-Campbell-Hausdorff product, and $\autoij{k}_{\alpha\beta, i_0 i_1}(\simplex_n)\in \mathcal{A}^0(\simplex_n) \otimes \bva{k}_{\beta}^{0}(U_{i_0 i_1})$ is the unique element such that
		$$
		\exp([\autoij{k}_{\alpha\beta,i_0i_1}(\simplex_n), \cdot]) \glue{k}_{\alpha \beta, i_0}(\simplex_n) = \glue{k}_{\alpha \beta, i_1}(\simplex_n).
		$$
	\end{enumerate}
\end{condition}

\begin{remark}\label{not:simplex_notation_simplify}
	Note that there are {\em two different} simplices in play in Notation \ref{not:local_thom_whitney_complex} and Condition \ref{assum:induction_hypothesis}: the $l$-simplex $\simplex_l$ in defining the Thom-Whitney complex $\twc{}^{p,q}(\bva{})$ in Definition \ref{def:thom_whitney_general}, and another $n$-simplex $\simplex_n$ in defining $\bva{}^*$ and $\twc{k}^{*,*}_{\alpha_i;\alpha_0 \cdots \alpha_{\ell}}(\simplex_n)$ in Notation \ref{not:local_thom_whitney_complex}. To simplify notations, we will suppress the dependence on $\simplex_n$ in the rest of this paper when there is no danger of confusion.
\end{remark}

Fixing $\simplex_n$ (which we omit the following notations), then for any triple $V_\alpha, V_\beta, V_\gamma \in \mathcal{V}$, we define the {\em restriction of $\glue{k}_{\alpha \beta}$ to $\twc{k}^{*,*}_{\alpha;\alpha\beta\gamma}$} as the unique map $\glue{k}_{\alpha \beta}:\twc{k}^{*,*}_{\alpha;\alpha\beta\gamma}\to \twc{k}^{*,*}_{\beta;\alpha\beta\gamma}$ that fits into the following diagram
$$
\xymatrix@1{ \twc{k}^{*,*}_{\alpha;\alpha\beta} \ar[r]^{\restmap_\gamma} \ar[d]^{\glue{k}_{\alpha \beta}} & \twc{k}^{*,*}_{\alpha;\alpha\beta\gamma} \ar[d]^{\glue{k}_{\alpha \beta}}\\
	\twc{k}^{*,*}_{\beta;\alpha\beta} \ar[r]^{\restmap_\gamma} & \twc{k}^{*,*}_{\beta;\alpha\beta\gamma}.}
$$

\begin{definition}\label{def:compatible_gluing_morphism}
	For a fixed $\simplex_n$, a collection $\glue{}(\simplex_n) = (\glue{k}_{\alpha \beta}(\simplex_n))_{k,\alpha\beta}$ (or simply  $\glue{} = (\glue{k}_{\alpha \beta})_{k,\alpha\beta}$ when the dependence on $\simplex_n$ is clear) satisfying Condition \ref{assum:induction_hypothesis} is said to be a {\em compatible gluing morphism over $\simplex_n$} if in addition the following conditions are satisfied:
	\begin{enumerate}
		\item $\glue{0}_{\alpha\beta} = \text{id}$ for all $\alpha, \beta$;
		
		\item (compatibility between different orders) for each $k \in \mathbb{N}$ and any pair $V_\alpha, V_\beta \in \mathcal{V}$,
		\begin{equation}\label{eqn:g_alpha_beta_order_compatibility}
		\glue{k}_{\alpha \beta} \circ \rest{k+1,k}_{\alpha} = \rest{k+1,k}_{\beta} \circ \glue{k+1}_{\alpha \beta};
		\end{equation}
		
		\item (cocycle condition) for each $k \in \mathbb{N}$ and any triple $V_\alpha, V_\beta, V_\gamma \in \mathcal{V}$,
		\begin{equation}\label{eqn:g_alpha_beta_cocycle}
		\glue{k}_{\gamma \alpha} \circ \glue{k}_{\beta \gamma} \circ \glue{k}_{\alpha \beta} = \text{id}
		\end{equation}
		when $\glue{k}_{\alpha \beta}$, $\glue{k}_{\beta \gamma}$ and $\glue{k}_{\gamma \alpha}$ are restricted to $\twc{k}^{*,*}_{\alpha;\alpha \beta \gamma}$, $\twc{k}^{*,*}_{\beta;\alpha \beta \gamma}$ and $\twc{k}^{*,*}_{\gamma;\alpha \beta \gamma}$ respectively. 
	\end{enumerate}
\end{definition}

For any $a : [m] \rightarrow [n]$ and the corresponding pull back $a^* : \mathcal{A}^*(\simplex_n) \rightarrow \mathcal{A}^*(\simplex_m)$, there is a naturally induced data $a^*(\glue{}(\simplex_n)) = (a^*(\glue{k}_{\alpha\beta}(\simplex_n)))_{k,\alpha\beta}$ as a compatible gluing morphism over $\simplex_{m}$.

\begin{definition}
	The {\em simplicial set of compatible gluing morphisms} $\gluesset(\simplex_{\bullet}) : \text{Mon} \rightarrow \text{Sets}$ is defined by letting $\gluesset(\simplex_n)$ be the set of compatible gluing morphisms over $\simplex_n$. 
\end{definition}

\subsubsection{The \v{C}ech-Thom-Whitney complex}\label{sec:the_Cech_thom_whitney_complex}
Given a compatible gluing morphism $\glue{}(\simplex_n)$ over $\simplex_n$, we construct a \v{C}ech-Thom-Whitney complex $\cech{k}^*(\twc{}, \glue{},\simplex_n)$ (or simply $\cech{k}^*(\twc{}, \glue{})$ when the dependence on $\simplex_n$ is clear) for each $k\in \mathbb{N}$.

\begin{definition}\label{def:cech_thom_whitney_complex}
	For each $\simplex_n$ and $\ell \in \mathbb{N}$, we let $\twc{k}^{*,*}_{\alpha_0 \cdots \alpha_{\ell}}(\glue{},\simplex_n) \subset \bigoplus_{i=0}^\ell \twc{k}^{*,*}_{\alpha_i;\alpha_0 \cdots \alpha_{\ell}}(\simplex_n)$ be the set of elements $(\varphi_0,\cdots, \varphi_{\ell})$ such that $ \varphi_{j} = \glue{k}_{\alpha_i \alpha_j}(\simplex_n) (\varphi_i)$. Then the {\em $k$-th order \v{C}ech-Thom-Whitney complex} $\cech{k}^*(\twc{}^{*,*}, \glue{},\simplex_n)$ over $X$ is defined by
	$$\cech{k}^\ell(\twc{}^{p,q},\glue{},\simplex_n) := \prod_{\alpha_0 \cdots \alpha_\ell} \twc{k}^{p,q}_{\alpha_0 \cdots \alpha_{\ell}}(\glue{},\simplex_n)$$
	and $\cech{k}^\ell(\twc{}^{*,*},\glue{},\simplex_n):= \bigoplus_{p,q} \cech{k}^\ell(\twc{}^{p,q},\glue{})$
	for each $k \in \mathbb{N}$. It is equipped with the {\em \v{C}ech differential} $\cechd{k}_\ell := \sum_{j =0}^{\ell+1} (-1)^j\restmap_{j,\ell+1}: \cech{k}^{\ell} (\twc{},\glue{},\simplex_n) \to \cech{k}^{\ell+1}(\twc{},\glue{},\simplex_n)$, where $\restmap_{j,\ell} : \cech{k}^{\ell-1}(\twc{},\glue{},\simplex_n) \rightarrow \cech{k}^{\ell}(\twc{},\glue{},\simplex_n)$ is the natural restriction map defined component-wise by $\restmap_{j,\ell}: \twc{k}^{*,*}_{\alpha_0 \cdots \widehat{\alpha_j} \cdots  \alpha_{\ell}}(\glue{},\simplex_n) \rightarrow \twc{k}^{*,*}_{\alpha_0 \cdots \alpha_{\ell}}(\glue{},\simplex_n)$ coming from Notation \ref{not:local_thom_whitney_complex}.
\end{definition}

For each $\simplex_n$, we set $\polyv{k}^{*,*}(\glue{},\simplex_n) := \ker(\cechd{k}_{0})$ (or simplified as $\polyv{k}^{*,*}(\glue{})$ when the dependence of $\simplex_n$ is clear). We denote the natural inclusion $\polyv{k}^{*,*}(\glue{}) \rightarrow \cech{k}^0(\twc{}^{*,*},\glue{})$ by $\cechd{k}_{-1}$, so we have the following sequence of maps
\begin{equation}\label{eqn:cech_thom_whitney_complex}
0 \rightarrow \polyv{k}^{p,q}(\glue{}) \rightarrow \cech{k}^0(\twc{}^{p,q},\glue{}) \rightarrow \cech{k}^1(\twc{}^{p,q},\glue{}) \rightarrow \cdots \rightarrow \cech{k}^{\ell}(\twc{}^{p,q},\glue{}) \rightarrow \cdots.
\end{equation}

For each $\simplex_n$, $\ell\in \mathbb{N}$ and $k\geq l$, there is a natural map $\rest{k,l}: \cech{k}^\ell(\twc{}^{p,q},\glue{})  \rightarrow \cech{l}^{\ell}(\twc{}^{p,q},\glue{})$ defined component-wise by the map $\rest{k,l}_{\alpha_j}: \twc{k}^{p,q}_{\alpha_j;\alpha_0 \cdots \alpha_{\ell}} \rightarrow \twc{l}^{p,q}_{\alpha_j;\alpha_0 \cdots \alpha_{\ell}}$ obtained from $\rest{k,l}_{\alpha}: \bva{k}_{\alpha}^* \rightarrow \bva{l}_{\alpha}^*$ in Definition \ref{def:abstract_deformation_data}. We then define $\cech{}^\ell(\twc{}^{p,q},\glue{}) :=\varprojlim_k \cech{k}^{\ell}(\twc{}^{p,q},\glue{})$ as the inverse limit along these maps and set $\cech{}^{\ell}(\twc{},\glue{}) := \bigoplus_{p,q} \cech{}^\ell(\twc{}^{p,q},\glue{})$. Similarly, we have the natural maps $\rest{k,l}:\polyv{k}^{p,q}(\glue{}) \rightarrow \polyv{l}^{p,q}(\glue{})$, and we can define $\polyv{}^{p,q}(\glue{})$ also by inverse limits and set $\polyv{}^{*,*}(\glue{}) := \bigoplus_{p,q} \polyv{}^{p,q}(\glue{})$.

For each $a : [m] \rightarrow [n]$ with the corresponding pullback $a^* : \mathcal{A}^*(\simplex_n) \rightarrow \mathcal{A}^*(\simplex_m)$, there are naturally induced maps $a^* : \cech{k}^\ell(\twc{}^{p,q},\glue{},\simplex_n) \rightarrow \cech{k}^\ell(\twc{}^{p,q},\glue{},\simplex_m)$ and $a^* :\polyv{k}^{p,q}(\glue{},\simplex_n) \rightarrow \polyv{k}^{p,q}(\glue{},\simplex_m)$ defined using the induced gluing morphisms $a^*(\glue{}(\simplex_n))$. 

\begin{definition}\label{def:compatible_differential}
	Fixing a compatible gluing morphism $\glue{}(\simplex_n) = (\glue{k}_{\alpha \beta}(\simplex_n))$ over $\simplex_n$, a {\em compatible differential} is an element $\dbtwist{}(\simplex_n) = \varprojlim_k \dbtwist{k}(\simplex_n)$ (or $\dbtwist{} = \varprojlim_k \dbtwist{k}$ by dropping its dependence on $\simplex_n$), where $\dbtwist{k}(\simplex_n)= (\dbtwist{k}_{\alpha}(\simplex_n))_{\alpha} \in \cech{k}^\ell(\twc{}^{1},\glue{},\simplex_n) $, such that, for each $k$, we have
	\begin{equation}\label{eqn:compatible_differential}
	 \glue{k}_{\beta\alpha}(\simplex_n) \circ (\prescript{k}{}{\pdb}_\beta +\prescript{k}{}{d}_{\beta,\simplex_n} + [\dbtwist{k}_{\beta}(\simplex_n), \cdot])  \circ \glue{k}_{\alpha\beta}(\simplex_n)  = \prescript{k}{}{\pdb}_{\alpha}+\prescript{k}{}{d}_{\alpha,\simplex_n} + [\dbtwist{k}_{\alpha}(\simplex_n),\cdot].
	\end{equation}
	For each $a : [m] \rightarrow [n]$, there is an induced compatible differential $a^*(\dbtwist{}(\simplex_n))$ for $a^*(\glue{}(\simplex_n))$.
\end{definition}

\begin{definition}
	We define the {\em simplicial set of compatible morphisms and differentials} $\gluediff(\simplex_{\bullet}) : \text{Mon} \rightarrow \text{Sets}$ by setting $\gluediff(\simplex_n) := \{ \left(\glue{}(\simplex_n), \dbtwist{}(\simplex_n)\right) \}$, where $\dbtwist{}(\simplex_n)$ is a compatible differential for a compatible gluing morphism $\glue{}(\simplex_n)$. 
\end{definition}

There is a natural morphism of simplicial sets $\gluediff(\simplex_{\bullet}) \rightarrow \gluesset(\simplex_{\bullet})$ defined by forgetting the differential. The results in \cite[\S 3.3-3.5]{chan2019geometry} can now be summarized as follows:
\begin{lemma}[\cite{chan2019geometry}]\label{lem:gluing_data_contractible}
	The simplicial sets $\gluesset(\simplex_{\bullet})$ and $\gluediff(\simplex_{\bullet})$ are non-empty (meaning that $\gluesset(\simplex_0) \neq \emptyset$ and $\gluediff(\simplex_0) \neq \emptyset$) and contractible (i.e., given an element $g = (g_0,\dots,g_n)$ in the boundary $\gluesset(\simplexbdy_n):= \{(g_0,\dots,g_n) \mid g_i \in \gluesset(\simplex_{n-1}), \ \mathtt{d}^*_{i,n-1}(g_j) = \mathtt{d}^*_{j,n-1}(g_i) \}$, there exists $\tilde{g} \in \gluesset(\simplex_n)$ such that $\mathtt{d}_{j,n}^*(\tilde{g}) = g_j$, and similarly for $\gluediff(\simplex_{\bullet})$).
\end{lemma}

With a compatible differential $\dbtwist{}$ over $\simplex_n$, the local operators $(\pdb_{\alpha} + d_{\alpha}+ [\dbtwist{}_{\alpha},\cdot])_{\alpha}$ glue to give a global differential operator $\mathbf{d}_{\dbtwist{}} $ on $\polyv{}^{*}(\glue{})$ while the Lie brackets glue together automatically, thus producing a dgLa. 

Consider the category $\mathtt{A}_{\cfr}$ of local Artinian $\cfr$-algebras with residue field $\comp$. With $(A,\mathbf{m}_A) \in \mathtt{A}_{\cfr}$ and $(\glue{}(\simplex_n),\dbtwist{}(\simplex_n)) \in \gluediff(\simplex_n)$, we let $\polyv{A}^*(\glue{},\simplex_n) := \polyv{k}^*(\glue{},\simplex_n) \otimes_{(\cfrk{k})} A$ (or simply $\polyv{A}^*(\glue{})$) for large enough $k$ such that $\mathbf{m}^{k+1} \cdot A \subset \{0\}$ and equip it with the differential $\prescript{A}{}{\mathbf{d}}_{\simplex_n}$ (or simply $\prescript{A}{}{\mathbf{d}}$). We use $\rest{A,B} : \polyv{A}^*(\glue{}) \rightarrow \polyv{B}^*(\glue{})$ to denote the naturally induced map from a morphism $A \rightarrow B$ in $\mathtt{A}_{\cfr}$. 

\begin{definition}
	For a fixed $\simplex_n$, an element $\varphi \in \polyv{A}^1(\glue{})$ such that $\varphi = 0 \ (\text{mod $\mathbf{m}_A$})$ is called a {\em Maurer-Cartan element} if it satisfies the Maurer-Cartan equation:
	\begin{equation}\label{eqn:maurer_cartan_equation}
		\mathbf{d}_{\dbtwist{}} \varphi + \half [\varphi,\varphi] = 0.
	\end{equation}
\end{definition}

Given $a : [m] \rightarrow [n]$ and any Maurer-Cartan element $\varphi$, $a^*(\varphi)$ is also a Maurer-Cartan element (with respect to $a^*(\glue{}(\simplex_n), \dbtwist{}(\simplex_n))$). Following \cite{getzler2009lie, hinich1996descent}, we define the {\em Maurer-Cartan simplicial set} $\mcsimplicial{A}(\simplex_{\bullet})$ over $A$ by setting
$$\mcsimplicial{A}(\simplex_n) := \left\{ \left(\glue{}(\simplex_n), \dbtwist{}(\simplex_n), \varphi\right) \mid (\glue{}(\simplex_n), \dbtwist{}(\simplex_n)) \in \gluediff(\simplex_n), \ \text{$\varphi$ satisfies \eqref{eqn:maurer_cartan_equation}} \right\}.$$ 
The following lemma is parallel to the results from \cite{getzler2009lie, hinich1996descent}.
\begin{lemma}\label{lem:kan_complex_lemma}
	The Maurer-Cartan simplicial set $\mcsimplicial{A}(\simplex_{\bullet})$ is a Kan complex. 
\end{lemma}
\begin{proof}
	Given $(\glue{}_i,\dbtwist{}_i,\prescript{A}{}{\varphi}_{i}) \in  \mcsimplicial{A}(\simplex_{n-1})$ for $0 \leq i \leq n$ and $i\neq k$ such that $\mathtt{d}_{i,n-1}^*(\glue{}_j,\dbtwist{}_j,\prescript{A}{}{\varphi}_{j}) = \mathtt{d}_{j-1,n-1}^*(\glue{}_i,\dbtwist{}_i,\prescript{A}{}{\varphi}_{i})$ for $ 0 \leq i < j \leq n$ and $i,j \neq k$, we need to construct $(\glue{},\dbtwist{},\prescript{A}{}{\varphi}) \in \mcsimplicial{A}(\simplex_{n})$ so that $\mathtt{d}_{i,n}^*(\glue{},\dbtwist{},\prescript{A}{}{\varphi} ) = (\glue{}_i,\dbtwist{}_i,\prescript{A}{}{\varphi}_{i})$. The existence of $(\glue{},\dbtwist{})$ follows from the contractibility of $\gluediff(\simplex_\bullet)$. For $\prescript{A}{}{\varphi}$, we assume that $\mathtt{h}: A \rightarrow B$ is a small extension and that $\prescript{B}{}{\varphi}$ has already been constructed such that $\mathtt{d}_{i,n}^*(\prescript{B}{}{\varphi}) = \rest{A,B}(\prescript{A}{}{\varphi}_{i})$. By the discussion in \cite[\S 3.4]{chan2019geometry}, one can always construct $\prescript{A}{}{\hat{\varphi}} \in \polyv{A}^1(\glue{},\simplex_n)$ such that $\mathtt{d}_{i,n}^*(\prescript{A}{}{\hat{\varphi}}) = \prescript{A}{}{\varphi}_{i}$. Therefore $\mathbf{d}_{\simplex_n} (\prescript{A}{}{\hat{\varphi}}) + \half [\prescript{A}{}{\hat{\varphi}},\prescript{A}{}{\hat{\varphi}}] = \prescript{A}{}{\mathtt{O}}$ with $[\prescript{A}{}{\mathtt{O}}] \in \mathbb{H}^2(\mathcal{A}^*(\simplex_n) \otimes \bva{0}^*) \otimes_{\comp} (\mathbf{m}_{A}/\ker(\mathtt{h}))$ (note that the local sheaves $\bva{0}^*_{\alpha}$'s glue to give a global sheaf $\bva{0}^*$ over $X$). Now the natural restriction map $\mathcal{A}^*(\simplex_n) \rightarrow \mathcal{A}^*(\Lambda^k_n)$ to the $k$-th horn $\Lambda^{k}_n$ is a quasi-isomorphism, thus inducing an isomorphism $\mathbb{H}^2(\mathcal{A}^*(\simplex_n) \otimes \bva{0}^*) \rightarrow \mathbb{H}^2(\mathcal{A}^*(\Lambda^k_n) \otimes \bva{0}^*)$. As a result, we have the obstruction class $[\prescript{A}{}{\mathtt{O}}] = 0$, giving the desired $\prescript{A}{}{\varphi}_{\simplex_n}$.
\end{proof}

Given an algebra homomorphism $c:A \rightarrow B$, there is a natural map $c(\simplex_{\bullet}) : \mcsimplicial{A}(\simplex_{\bullet}) \rightarrow \mcsimplicial{B}(\simplex_{\bullet})$ of simplicial sets. These can be packed together as the {\em simplicial Maurer-Cartan functor} $\mcsimplicial{\bullet}(\simplex_{\bullet}) : \mathtt{A}_{\cfr} \rightarrow \text{sSet}$ from $\mathtt{A}_{\cfr}$ to the category $\text{sSet}$ of simplicial sets. By taking the connected component $\pi_0 (\mcsimplicial{A}(\simplex_\bullet))$ of simplicial sets, we obtain the ordinary Maurer-Cartan functor $ \pi_0 (\mcsimplicial{\bullet}(\simplex_{\bullet})) $ defined by $A \mapsto \pi_0(\mcsimplicial{A}(\simplex_\bullet))$.

\subsubsection{Morphism of dgLa's}\label{sec:construction_morphism_of_dgla_from_gluing}

A morphism $\morph{} = \{\morph{k}_{\alpha}\}_{k,\alpha} :  \gla{}=\{\gla{k}^*_{\alpha}\}_{k,\alpha}  \rightarrow \bva{}=\{\bva{k}^*_{\alpha}\}_{k,\alpha}$ of patching data (Definition \ref{def:morphism_of_deformation_data}) naturally induces a morphism of dgLa's on the corresponding local Thom-Whitney complexes $\morph{k}_{\alpha_i}: \twcp{k}_{\alpha_i;\alpha_0\cdots\alpha_{\ell}}^{*,*}(\simplex_n)\rightarrow \twc{k}_{\alpha_i;\alpha_0\cdots\alpha_{\ell}}^{*,*}(\simplex_n)$
constructed from component-wise maps $\morph{k}_{\alpha_i} : \mathcal{A}^*(\simplex_{l})\otimes \mathcal{A}^*(\simplex_{n}) \otimes \gla{k}_{\alpha_i}^*(U_{i_0\cdots i_l}) \rightarrow  \mathcal{A}^*(\simplex_{l})\otimes \mathcal{A}^*(\simplex_{n}) \otimes \bva{k}_{\alpha_i}^*(U_{i_0\cdots i_l})$.

\begin{lemma}\label{lem:dgla_morphism_checking}
	Given a compatible gluing morphism $\gluep{} = (\gluep{k}_{\alpha\beta})_{k,\alpha\beta}$ for $\gla{}$ which satisfies Condition \ref{assum:induction_hypothesis} with the elements $\iautop{k}_{\alpha\beta,i}$'s and $\sautop{k}_{\alpha\beta,i_0\cdots i_l}$'s over $\simplex_n$, we define $\glue{} := \morph{}\left( \gluep{} \right)$ over $\simplex_n$ by setting $\iauto{k}_{\alpha\beta,i} := \morph{k}_{\beta}(\iautop{k}_{\alpha\beta,i})$ and $\sauto{k}_{\alpha\beta,i_0\cdots i_l} := \morph{k}_{\beta}(\sautop{k}_{\alpha\beta,i_0\cdots i_l})$. Then $\glue{}$ is a compatible gluing morphism for $\bva{}$ over $\simplex_n$.
\end{lemma}

\begin{proof}
	 From its construction, we see that $\glue{k}_{\alpha\beta} \circ \morph{k}_{\alpha} = \morph{k}_{\beta} \circ \gluep{k}_{\alpha\beta}$. With $\autoijp{k}_{\alpha\beta,ij}$'s and $\autoij{k}_{\alpha\beta,ij}$'s in Condition \ref{assum:induction_hypothesis}, we observe the relation
	$$
	 \exp \left( [\autoijp{k}_{\alpha\beta,ij},\cdot] \right)  = \exp\left([\iautop{k}_{\alpha\beta,j},\cdot] \right) \circ  \exp \left( [\patchijp{k}_{\beta\alpha,ji} , \cdot ] \right)  \circ \exp\left([-\iautop{k}_{\alpha\beta,i},\cdot] \right),
	$$
	which gives $\autoijp{k}_{\alpha\beta,ij} = \iautop{k}_{\alpha\beta,j} \bchprod \patchijp{k}_{\beta\alpha,ji} \bchprod (-\iautop{k}_{\alpha\beta,i}) $ by the injectivity of the map $\gla{k}_{\beta}^0 \hookrightarrow \der{(\cfrk{k})}{\gla{k}_{\beta}^*}{0}$, 
	and similarly for the elements $\autoij{k}_{\alpha\beta,ij}$'s, $\iauto{k}_{\alpha\beta,i}$'s and $\patchij{k}_{\beta\alpha,ji}$'s. So we have $\morph{k}_{\beta}(\autoijp{k}_{\alpha\beta,ij}) = \autoij{k}_{\alpha\beta,ij}$, which shows that $\sauto{k}_{\alpha\beta,i_0\cdots i_l}$ satisfies the relation \eqref{eqn:alpha_beta_gluing_explicit_form_on_simplex_boundary_relation}. 
	
	For the cocycle condition in Definition \ref{def:compatible_gluing_morphism}, we first consider the elements $\glue{k}_{\alpha\beta,i}$'s. A direct computation gives
	\begin{multline*}
	\glue{k}_{\gamma \alpha, i} \circ \glue{k}_{ \beta \gamma, i} \circ \glue{k}_{\alpha \beta, i} = \\ \exp([\iauto{k}_{\gamma\alpha,i},\cdot]) \circ \exp([\patch{k}_{\gamma\alpha,i} (\iauto{k}_{\beta\gamma,i}),\cdot]) \circ \exp([\patch{k}_{\gamma\alpha,i}\circ \patch{k}_{\beta\gamma,i} (\iauto{k}_{\alpha\beta,i}),\cdot]) \circ \patch{k}_{\gamma\alpha,i} \circ \patch{k}_{\beta\gamma,i} \circ \patch{k}_{\alpha\beta,i},
	\end{multline*}
	and the cocycle condition is equivalent (under the injection $\bva{k}_{\beta}^0 \hookrightarrow \der{(\cfrk{k})}{\bva{k}_{\beta}^*}{0}$) to 
	$$
	\iauto{k}_{\gamma\alpha,i} \bchprod \patch{k}_{\gamma\alpha,i} (\iauto{k}_{\beta\gamma,i}) \bchprod \patch{k}_{\gamma\alpha,i}\circ \patch{k}_{\beta\gamma,i} (\iauto{k}_{\alpha\beta,i}) \bchprod  \cocyobs{k}_{\alpha\beta\gamma,i}  = 0.
	$$
	Since $\morph{k}_{\alpha}$ is a morphism of dgLa's, $\glue{k}_{\alpha\beta,i}$'s satisfy the cocycle condition since $\gluep{k}_{\alpha\beta,i}$'s do. For general $\glue{k}_{\alpha\beta,i_0 \cdots i_l}$'s, notice that the cocycle condition is equivalent to 
	$$
	\sauto{k}_{\gamma\alpha,i_0 \cdots i_l} \bchprod  \glue{k}_{\gamma \alpha, i_0}(\sauto{k}_{\beta \gamma,i_0 \cdots i_l}) \bchprod \glue{k}_{\beta \alpha, i_0}(\sauto{k}_{\alpha \beta,i_0 \cdots i_l}) = 0,
	$$
	which also follows from the fact that $\morph{k}_{\alpha}$ preserves the graded Lie algebra structure. The condition $(2)$ in Definition \ref{def:compatible_gluing_morphism} can be proved similarly.
\end{proof}

Now for a map $a : [m] \rightarrow [n]$ in $\text{Mon}$, we have $\morph{} \left( a^*(\gluep{}) \right) = a^* \left( \morph{} (\gluep{}) \right)$ as compatible gluing morphisms over $\simplex_m$. Hence $\morph{}$ induces a morphism $\morph{} : \gluessetp \rightarrow \gluesset$ of simplicial sets, where $\gluessetp$ is the one associated to $\gla{}$. 

A morphism $\morph{}$ also induces a map $\morph{k} : \cech{k}^\ell(\twcp{},\gluep{}) \rightarrow \cech{k}^\ell(\twc{},\glue{})$ between the \v{C}ech-Thom-Whitney complexes which is naturally compatible with $\rest{k,l}$'s and $\cechd{k}$'s (here we use the notations $\rest{k,l}$ and $\cechd{k}$ for both $\gla{}$ and $\bva{}$). Suppose we have a compatible differential $\dbtwistp{}$ for $\gluep{}$ over $\simplex_n$, then $\dbtwist{} :=\morph{}(\dbtwistp{}) := \varprojlim_k \morph{k}(\dbtwistp{k})$ is a compatible differential for $\morph{}\left( \gluep{}\right)$ over $\simplex_n$. Therefore, it gives a morphism $\morph{} : \polyv{}^*(\gluep{}) \rightarrow \polyv{}^*(\glue{})$ of the corresponding dgLa's. In this way $\morph{}$ also induces maps $\morph{} :\gluediffp(\simplex_{\bullet}) \rightarrow \gluediff(\simplex_{\bullet})$ and $\morph{A} : \mcsimplicialp{A}(\simplex_{\bullet}) \rightarrow \mcsimplicial{A}(\simplex_{\bullet})$ of simplicial sets, for every $A \in \mathtt{A}_{\cfr}$. The latter morphisms can be packed together to give a natural transformation $\morph{}: \mcsimplicialp{\bullet}(\simplex_{\bullet})  \rightarrow  \mcsimplicial{\bullet}(\simplex_{\bullet})$.

\begin{definition}[Definition 1.37 in \cite{pridham2010unifying}]\label{def:smoothness_of_deformation_functor}
	We say the Maurer-Cartan functor $\mcsimplicial{\bullet}(\simplex_{\bullet})$ (and similarly for $\mcsimplicialp{\bullet}(\simplex_{\bullet})$) is {\em smooth} if the natural restriction map $\rest{A,B} : \mcsimplicial{A}(\simplex_{\bullet}) \rightarrow \mcsimplicial{B}(\simplex_{\bullet})$ is a surjective Kan fibration (see \cite[Fibrations 8.2.9]{weibel1995introduction} for a definition of Kan fibration).
	
	With $\morph{}:\gla{} \rightarrow \bva{}$ as above, we say the associated natural transformation $\morph{} : \mcsimplicialp{\bullet}(\simplex_{\bullet})  \rightarrow  \mcsimplicial{\bullet}(\simplex_{\bullet})$ is {\em smooth} if, 
	for every surjective map $A \rightarrow B$ in $\mathtt{A}_{\cfr}$, the naturally induced map 
	\begin{equation}\label{eqn:mc_functor_smoothness}
	\morph{A,B}: \mcsimplicialp{A}(\simplex_{\bullet}) \rightarrow 	\mcsimplicialp{B}(\simplex_{\bullet}) \times_{(\mcsimplicial{B}(\simplex_{\bullet}))} \mcsimplicial{A}(\simplex_{\bullet})
	\end{equation}
	is a surjective Kan fibration. 
\end{definition}

The above smoothness condition will imply the smoothness of the ordinary Maurer-Cartan functor $\pi_0(\mcsimplicial{\bullet}(\simplex_{\bullet}))$, and similarly for the natural transformation $\pi_0(\morph{}) : \pi_0(\mcsimplicialp{\bullet}(\simplex_{\bullet}))  \rightarrow \pi_0( \mcsimplicial{\bullet}(\simplex_{\bullet}))$.

The local sheaves $\gla{0}_{\alpha}^*$'s and $\bva{}_{\alpha}^*$'s glue to give global sheaves  $\gla{0}$ and $\bva{0}$ of dgLa's over $X$, and the morphisms $\morph{0}_{\alpha}$'s glue to give a morphism $\morph{0} : \gla{0}_{\alpha}^* \rightarrow \bva{0}_{\alpha}^*$ of sheaves of dgLa's. Let $\kla := \ker(\morph{0})$ be the kernel subsheaf. Then both $\polyv{0}^*(\gluep{},\simplex_n)$ and $\polyv{0}^*(\glue{},\simplex_n)$ are independent of the choice of the gluing morphism and compatible differential, and they are quasi-isomorphic to $\twc{}^*(\gla{0}) \otimes_\comp \mathcal{A}^*(\simplex_n)$ and $\twc{}^*(\bva{0}) \otimes_\comp \mathcal{A}^*(\simplex_n)$ respectively. If we further assume that each $\morph{k}_{\alpha} : \gla{k}_{\alpha}^* \rightarrow \bva{k}_{\alpha}^*$ is a surjective morphism of sheaves, then we obtain the following exact sequence of dgLa's:
$$
\xymatrix@1{
 0  \ar[r] & \twc{}^*(\kla)\otimes_\comp \mathcal{A}^*(\simplex_n) \ar[r] \ar[d]^{\cong} &  \twc{}^*(\gla{0}) \otimes_\comp \mathcal{A}^*(\simplex_n) \ar[r]  \ar[d]^{\cong} &  \twc{}^*(\bva{0}) \otimes_\comp \mathcal{A}^*(\simplex_n) \ar[r] \ar[d]^{\cong} & 0 \\
  0  \ar[r] & \prescript{0}{}{\mathscr{K}}^*(\simplex_n) \ar[r] & \polyv{0}^*(\gluep{},\simplex_n) \ar[r] & \polyv{0}^*(\glue{},\simplex_n)  \ar[r] & 0,
}
$$
where $\prescript{0}{}{\mathscr{K}}^*(\simplex_n)$ is the kernel dgLa of the natural map $\polyv{0}^*(\gluep{},\simplex_n) \rightarrow \polyv{0}^*(\glue{},\simplex_n)$ and the vertical arrows are all quasi-isomorphisms of dgLa's.

\begin{theorem}\label{thm:main_theorem}
	Suppose $\morph{} : \gla{} \rightarrow \bva{}$ is a morphism  between abstract deformation data such that each $\morph{k}_{\alpha} : \gla{k}_{\alpha}^* \rightarrow \bva{k}_{\alpha}^*$ is a surjective morphism of sheaves and the hypercohomology group $\mathbb{H}^2(X,\kla^*) = 0$. Then $\morph{} : \mcsimplicialp{\bullet}(\simplex_{\bullet})  \rightarrow  \mcsimplicial{\bullet}(\simplex_{\bullet})$ is smooth. In particular, $\mcsimplicialp{\bullet}(\simplex_{\bullet})$ is smooth if $\mcsimplicial{\bullet}(\simplex_{\bullet})$ is smooth
\end{theorem}

\begin{proof}
	It suffices to consider a small extension $\mathtt{h}: A \rightarrow B$. A general element in the target of the map \eqref{eqn:mc_functor_smoothness} can be represented by Maurer-Cartan solutions $(\prescript{B}{}{\tilde{\varphi}}_{\simplex_n}, \prescript{A}{}{\varphi}_{\simplex_n})$ with respect to the gluing morphisms $(\gluep{}(\simplex_n),\glue{}(\simplex_n))$ and compatible differentials $(\dbtwistp{}(\simplex_n),\dbtwist{}(\simplex_n))$, where $(\glue{}(\simplex_n),\dbtwist{}(\simplex_n))= \morph{}(\gluep{}(\simplex_n),\dbtwistp{}(\simplex_n))$ and $\morph{}(\prescript{B}{}{\varphi}_{\simplex_n}) = \rest{A,B} (\prescript{A}{}{\varphi}_{\simplex_n})$. Given Maurer-Cartan solutions $\prescript{A}{}{\tilde{\varphi}}_i \in \polyv{A}^1(\gluep{},\simplex_{n-1})$ for $0 \leq i \leq n$ and $i\neq k$ with respect to the gluing morphism $\mathtt{d}_{i,n}^*(\gluep{}(\simplex_n))$ and compatible differential $\mathtt{d}_{i,n}^*(\dbtwistp{}(\simplex_n))$ such that $ \mathtt{d}_{i,n}^*(\prescript{B}{}{\tilde{\varphi}}_{\simplex_n}) = \rest{A,B}(\prescript{A}{}{\tilde{\varphi}}_i)$ and $\mathtt{d}_{i,n-1}^*(\prescript{A}{}{\tilde{\varphi}}_{j}) = \mathtt{d}_{j-1,n-1}^*(\prescript{A}{}{\tilde{\varphi}}_{i})$, we need to construct an element $\prescript{A}{}{\hat{\varphi}}_{\simplex_n} \in \polyv{A}^1(\gluep{},\simplex_n)$ lifting $\prescript{B}{}{\tilde{\varphi}}_{\simplex_n}$, and satisfying $\morph{}(\prescript{A}{}{\hat{\varphi}}_{\simplex_n}) = \prescript{A}{}{\varphi}_{\simplex_n}$ and $\mathtt{d}_{i,n}^*(\prescript{A}{}{\hat{\varphi}}_{\simplex_n}) = \prescript{A}{}{\tilde{\varphi}}_i$. 
	
 	By the discussion in \cite[\S 3.4]{chan2019geometry}, we can construct a lifting $\prescript{A}{}{\mathtt{u}}_{\simplex_n} \in \polyv{A}^1(\gluep{},\simplex_n)$ of $\prescript{B}{}{\tilde{\varphi}}_{\simplex_n}$ satisfying $\morph{}(\prescript{A}{}{\mathtt{u}}_{\simplex_n}) = \prescript{A}{}{\varphi}_{\simplex_n}$. Letting $\prescript{A}{}{\mathtt{u}}_i = \mathtt{d}_{i,n}^*(\prescript{A}{}{\mathtt{u}}_{\simplex_n})$ for $i \neq k$, we have $\rest{A,B}(\prescript{A}{}{\tilde{\varphi}}_i - \prescript{A}{}{\mathtt{u}}_i) = 0 $ and $\morph{A}(\prescript{A}{}{\tilde{\varphi}}_i - \prescript{A}{}{\mathtt{u}}_i) = 0$ which implies that $\prescript{A}{}{\mathtt{w}}_i :=   \prescript{A}{}{\tilde{\varphi}}_i - \prescript{A}{}{\mathtt{u}}_i$ lies in $\twc{}^{*,*}(\kla) \otimes_{\comp} \mathcal{A}^*(\simplex_{n-1}) \otimes_{\comp} (\mathbf{m}_A/\ker(\mathtt{h}))$. We can construct $\prescript{A}{}{\mathtt{v}} \in \twc{}^{*,*}(\kla) \otimes_{\comp} \mathcal{A}^*(\simplex_{n}) \otimes_{\comp} (\mathbf{m}_A/\ker(\mathtt{h}))$ such that $\mathtt{d}_{i,n}^*(\prescript{A}{}{\mathtt{v}}) = \prescript{A}{}{\mathtt{w}}_i$ and therefore the modification $\prescript{A}{}{\hat{\varphi}}_{\simplex_n}  = \prescript{A}{}{\mathtt{u}}  +\prescript{A}{}{\mathtt{v}}$ will have the desired properties.
	
	We have $\prescript{A}{}{\mathbf{d}}_{\simplex_n} (\prescript{A}{}{\hat{\varphi}}_{\simplex_n}) + \half [\prescript{A}{}{\hat{\varphi}}_{\simplex_n},\prescript{A}{}{\hat{\varphi}}_{\simplex_n}] = \prescript{A}{}{\mathtt{O}}_{\simplex_n} \in \twc{}^{*,*}(\kla) \otimes_{\comp} \mathcal{A}^*(\simplex_{n}) \otimes_{\comp} (\mathbf{m}_A/\ker(\mathtt{h}))$  with $\prescript{0}{}{\mathbf{d}}_{\simplex_n} (\prescript{A}{}{\mathtt{O}}_{\simplex_n}) = 0$, and hence representing a cohomology class satisfying $ [(\prescript{A}{}{\mathtt{O}}_{\simplex_n})|_{\Lambda_n^{k}}] = 0$. Using the fact that the natural restriction map $\mathcal{A}^*(\simplex_n) \rightarrow \mathcal{A}^*(\Lambda^k_n)$ to the $k$-th horn is a quasi-isomorphism, we have $\prescript{A}{}{\mathtt{O}}_{\simplex_n} = \prescript{0}{}{\mathbf{d}}_{\simplex_n}(\prescript{A}{}{\mathtt{c}}_{\simplex_n})$. This allows us to define a Maurer-Cartan element $ \prescript{A}{}{\tilde{\varphi}}_{\simplex_n} = \prescript{A}{}{\hat{\varphi}}_{\simplex_n} - \prescript{A}{}{\mathtt{c}}_{\simplex_n} \in \mcsimplicialp{A}(\simplex_n)$ as desired. This completes the proof that the map in \eqref{eqn:mc_functor_smoothness} is a Kan fibration.
	
	Now it remains to show that the map $\morph{A,B} : \mcsimplicialp{A}(\simplex_{0}) \rightarrow 	\mcsimplicialp{B}(\simplex_{0}) \times_{(\mcsimplicial{B}(\simplex_{0}))} \mcsimplicial{A}(\simplex_{0})$ on the $0$-simplex $\simplex_0$ is surjective for a small extension $\mathtt{h}:A \rightarrow B$. Again an element in the target can be written as $(\prescript{B}{}{\tilde{\varphi}}_{\simplex_0}, \prescript{A}{}{\varphi}_{\simplex_0})$ with respect to the gluing morphisms $(\gluep{}(\simplex_0),\glue{}(\simplex_0))$ and compatible differentials $(\dbtwistp{}(\simplex_0),\dbtwist{}(\simplex_0))$, and what we need is a lifting $\prescript{A}{}{\tilde{\varphi}}_{\simplex_0}$ as an element of the domain. As above, we can construct a lifting $\prescript{A}{}{\hat{\varphi}}_{\simplex_n} \in \polyv{A}^1(\gluep{},\simplex_0)$ of $(\prescript{B}{}{\tilde{\varphi}}_{\simplex_0}, \prescript{A}{}{\varphi}_{\simplex_0})$ and get the obstruction $(\prescript{A}{}{\mathtt{O}}_{\simplex_n}) \in \twc{}^*(\kla)$ which represents a  cohomology class in $H^2(\twc{}^*(\kla))$. From the assumption, we have $H^2(\twc{}^*(\kla))  = \mathbb{H}^2(X,\kla^*) =0$, forcing the obstruction to be zero. This proves the smoothness of $\morph{A,B}$, and the second statement follows from the first one.
\end{proof}

\subsection{Weakened assumption in Definition \ref{def:abstract_deformation_data}}\label{sec:weaken_assumption}
In this subsection, we explain the modification needed when condition $(3)$ in Definition \ref{def:abstract_deformation_data} is not satisfied -- this is essential for application of the results of this section to smoothing of pairs in \S \ref{sec:pair_data_from_gross_siebert} because the dgLa in Definition \ref{def:Lie_algebra_of_pairs} indeed does not satisfy this condition. 

The condition $(3)$ in Definition \ref{def:abstract_deformation_data} can be weakened as follows: Consider the adjoint homomorphisms $\text{ad} : \bva{0}^* \rightarrow \der{}{\bva{0}^*}{*}$ and $\text{ad} : \bva{k}_{\alpha}^* \rightarrow \der{(\cfrk{k})}{\bva{k}_{\alpha}^*}{*}$, where $\der{(\cfrk{k})}{}{*}$ denotes the graded vector space of $\cfrk{k}$-linear derivations. Let $\text{ad}(\bva{0}^*)$ and $\text{ad}(\bva{k}_{\alpha}^*)$ be their image sheaves of dgLa's respectively, which are equipped with the naturally induced morphisms $\rest{k,l}_{\alpha} : \text{ad}(\bva{k}^*_{\alpha}) \rightarrow \text{ad}(\bva{l}^*_{\alpha})$. Then condition $(3)$ in Definition \ref{def:abstract_deformation_data} can be weakened to the following condition:\\

\noindent \textit{($3'$) $\text{ad}(\bva{0}^*), \{\text{ad}(\bva{k}^*_{\alpha})\}_{k,\alpha}, \{\rest{k,l}_{\alpha}\}_{k\geq l,\alpha}$ satisfy conditions $(1)$, $(2)$ and $(4)$ in Definition \ref{def:abstract_deformation_data}.}\\

With this weakened condition, Condition \ref{assum:induction_hypothesis} and Definition \ref{def:compatible_differential} need to be adjusted:
First, Condition \ref{assum:induction_hypothesis} should be modified by only requiring the elements $\iauto{k}_{\alpha\beta,i}(\simplex_n)$'s, $\sauto{k}_{\alpha\beta, i_0 \cdots i_l}(\simplex_n)$'s and $\autoij{k}_{\alpha\beta,i_0i_1}(\simplex_n)$'s to take values in $\text{ad}(\bva{k}_{\alpha}^*)$, and the equation \eqref{eqn:alpha_beta_gluing_explicit_form_on_simplex_boundary_relation} to hold in $\text{ad}(\bva{k}_{\alpha}^*)$. The gluing morphism $\glue{k}_{\alpha\beta}(\simplex_n)$ will remain well defined as it involves only Lie brackets with those elements.
Similarly, in Definition \ref{def:compatible_differential}, we only require the elements $\dbtwist{k}_{\alpha}(\simplex_n)$'s to take values in $\text{ad}(\bva{k}_{\alpha}^*)$, which is enough for making sense of the equation \eqref{eqn:compatible_differential}.

Only Lemmas \ref{lem:gluing_data_contractible} and \ref{lem:dgla_morphism_checking} involve the original condition $(3)$ in Definition \ref{def:abstract_deformation_data}. But we observe that, under the new condition $(3')$, their proofs are not altered at all:
Lemma \ref{lem:gluing_data_contractible}, which we extract from \cite{chan2019geometry}, involves solving for the elements $\iauto{k}_{\alpha\beta,i}(\simplex_n)$'s and $\sauto{k}_{\alpha\beta, i_0 \cdots i_l}(\simplex_n)$'s in Condition \ref{assum:induction_hypothesis}, and the elements $\dbtwist{k}_{\alpha}(\simplex_n)$'s in Definition \ref{def:compatible_differential}. Under condition $(3')$, these elements can only be defined with values in the sheaves $\text{ad}(\bva{k}^*_{\alpha})$'s, but this will be enough because we observe that the compatibility conditions in Definitions \ref{def:compatible_gluing_morphism} and \ref{def:compatible_differential} only involve the values of these elements in the sheaves $\text{ad}(\bva{k}^*_{\alpha})$'s.
Similarly, in the proof of Lemma \ref{lem:dgla_morphism_checking}, these elements only take values in $\text{ad}(\gla{k}^*_{\alpha})$'s and $\text{ad}(\bva{k}^*_{\alpha})$'s and the equality involving them will hold in $\text{ad}(\gla{k}^*_{\alpha})$'s and $\text{ad}(\bva{k}^*_{\alpha})$'s, but these are already enough for our purposes.

\section{Smoothing of pairs}\label{sec:deformation_of_pair_on_log}

In this section, we construct the abstract local deformation data $\gla{}$ for a pair $(X,\mathfrak{C}^*)$, which would come with a natural forgetful morphism $\morph{} : \gla{} \rightarrow \bva{}$ sending $\gla{}$ to the abstract local deformation data $\bva{}$ associated to $X$ constructed in \cite{Felten-Filip-Ruddat}.

Throughout this section, we fix $Q = \mathbb{N}$.
Let $X$ be a $d$-dimensional projective toroidal crossing space over $\comp$.
According to \cite[Definition 1.5]{Felten-Filip-Ruddat}, a toroidal crossing space is defined as an algebraic space over $\comp$ together with a sheaf of monoids $\mathcal{P}$ with global section $\mathbf{1} \in \Gamma(X, \mathcal{P})$ such that, locally at every point $x\in X$, there is a smooth map to the boundary divisor $D_x$ in the affine toric variety $V_x = \text{Spec }\comp[\mathcal{P}_x]$ mapping $\mathbf{1}_x$ to the monomial in $\mathcal{P}_x$ whose divisor is $D_x$ and so that $\mathcal{P}$ is isomorphic to the pullback of $\mathcal{P}_{V_x}$; here $\mathcal{P}_{V_x}$ is the sheaf of monoids defined by $\mathcal{P}_{V_x} := \underline{\mathcal{P}_x} / a^{-1}(\mathcal{O}_{V_x}^\times)$ where $\underline{\mathcal{P}_x}$ denotes the constant sheaf and $a: \underline{\mathcal{P}_x} \to \mathcal{O}_{V_x}$ is the map $p \mapsto z^p$.
This notion was introduced by Schr\"oer-Siebert \cite{Siebert-Schroer}.
We assume that the higher tangent sheaf $\mathcal{T}_X^1 = \mathcal{E}\text{xt}^1(\Omega_X,\mathcal{O}_X)$ is globally generated. 
Then \cite[Theorem 6.8 and Proposition 6.9]{Felten-Filip-Ruddat} furnish $X$ with a log structure, or what they call the structure of a log toroidal family over the standard $\mathbb{N}$-log point $\logsk{0}$. We will denote the log scheme by $X^{\dagger}$ if we want to emphasize its log structure.

Let $Z \subset X$ be the codimension $2$ singular locus of the log structure (i.e., $X^{\dagger}$ is log smooth away from $Z$), and write $j : X\setminus Z \rightarrow X$ for the inclusion. Also let $W^*_{X^{\dagger}/\logsk{0}} = j_* \Omega^*_{(X\setminus Z)^{\dagger}/\logsk{0}}$ be the push forward of the sheaf of relative log differential forms on $X\setminus Z$ over $\logsk{0}$. We further assume the Calabi-Yau condition, namely, $\omega_X \cong \mathcal{O}_X$ (which is equivalent to the condition that $W^d_{X^{\dagger}/\logsk{0}} \cong \mathcal{O}_X$ by \cite[Lemma 6.11.]{Felten-Filip-Ruddat}).

\subsection{Construction of abstract deformation data for a space}\label{sec:data_from_gross_siebert}
Here we recall the construction of the abstract deformation data $\sla{}$ from \cite[\S 8]{chan2019geometry} concerning the smoothing of $X$.\footnote{Note that it was denoted as $\bva{}$ (not $\sla{}$) in \cite{chan2019geometry}.}
\begin{notation}\label{not:local_model}
	Following \cite[\S 3]{Felten-Filip-Ruddat}, we take an elementary (log) toroidal crossing datum $(\mathcal{Q}\subset P,\mathcal{F})$ consisting of monoids $\mathcal{Q}, P$ with an injection $\mathcal{Q} \hookrightarrow P$ and a collection $\mathcal{F}$ of facets of $P$ containing all facets that do not contain $Q$. We have the corresponding analytic schemes $\mathbf{V} = \text{Spec}(\comp[P])^{an}$, which is equipped with the divisorial log structure induced from the toric divisor $V$ corresponding to $\mathcal{F}$, and $V= \text{Spec}(\comp[\mathcal{Q}])^{an}$, which is equipped with the pullback log structure from $\mathbf{V}$. 
	There is a log morphism $\pi : \mathbf{V}^{\dagger} \rightarrow \logs$, induced from the natural monoid morphism $Q = \mathbb{N} \rightarrow P$, and a fiber diagram of log analytic schemes
	\begin{equation}\label{eqn:local_model}
	\xymatrix@1{ V^{\dagger} \  \ar@{^{(}->}[rr] \ar[d] & & \mathbf{V}^{\dagger} \ar[d]^{\pi}\\
		\logsk{0} \  \ar@{^{(}->}[rr] & &\logs .
	}
	\end{equation}
\end{notation}

As described in \cite[\S 13]{Felten-Filip-Ruddat}, for every point $x \in X$, there are some monoids $\mathcal{Q}$ and $P$ as in Notation \ref{not:local_model}, from which we can construct $V \subset \mathbf{V}$, together with a neighborhood $V_x$ of $x$ which can be identified with an open subset $V_x \subset V$, where $V_x$ can further be chosen to be Stein. We fix an open covering $\mathcal{V}$ by Stein open subsets $V_{\alpha}$'s, where $V_{\alpha} = V_x$ for some $x \in X$ which comes with a local $k$-th order thickening $\prescript{k}{}{\mathbf{V}}^{\dagger}_\alpha$'s in $\mathbf{V}_{\alpha}^{\dagger}$ over $\logsk{k}$. We will abuse notations and write $j: V_{\alpha} \setminus Z \rightarrow V_{\alpha}$. 

Following \cite[\S 8]{chan2019geometry} (and using notations from \cite{Gross-Siebert-logI}), we construct an abstract deformation datum $\bva{}$ as follows:
\begin{definition}\label{def:deformation_data_from_gross_siebert}
	\begin{enumerate}
		\item the sheaf $\sla{0}^*$ of dgLa's is given by the push forward $\sla{0}^* := j_* ( \Theta_{X^{\dagger}/\logsk{0}})$ of the analytic sheaf of relative log vector fields concentrated at degree $0$, equipped with the natural Lie bracket;
		
		\item  for each $k \in \mathbb{Z}_{\geq 0}$ and $\alpha$, the sheaf $\sla{k}_{\alpha}^*$ of dgLa's is given by the push forward $\sla{k}_{\alpha}^* := j_* (\Theta_{\prescript{k}{}{\mathbf{V}}_{\alpha}^{\dagger}/\logsk{k}})$ of the analytic sheaf of relative log vector fields concentated at degree $0$, equipped with the natural Lie bracket; 
		
		\item for $k\geq l$ and each $\alpha$, the morphism $\rest{k,l}_{\alpha} : j_* (  \Theta_{\prescript{k}{}{\mathbf{V}}_{\alpha}^{\dagger}/\logsk{k}}) \rightarrow j_* (  \Theta_{\prescript{l}{}{\mathbf{V}}_{\alpha}^{\dagger}/\logsk{l}})$ is obtained from the isomorphism $j_* (  \Theta_{\prescript{k}{}{\mathbf{V}}_{\alpha}^{\dagger}/\logsk{k}}) \otimes_{(\cfrk{k})} \cfrk{l} \cong j_* (  \Theta_{\prescript{l}{}{\mathbf{V}}_{\alpha}^{\dagger}/\logsk{l}})$.
	\end{enumerate}
\end{definition}

Conditions $(1)-(2)$ in Definition \ref{def:abstract_deformation_data} can be easily checked from the definition of the data. Condition $(4)$ follows from the fact that the $\sla{k}_{\alpha}$'s are coherent sheaves (\cite[Lemma 2.4]{Felten-Filip-Ruddat}). To verify condition $(3)$, it suffices to consider $\sla{0}^*$ since $\sla{k}_{\alpha}^*$ is a sheaf of free $\cfrk{k}$-modules with $[\cdot,\cdot]$ being $\cfrk{k}$-linear. Moreover, condition $(3)$ is a local statement which can be checked on $X\setminus Z$ away from the singular locus $Z$. Since $X \setminus Z$ is log smooth over $\logsk{0}$, this can be done directly on log smooth charts covering $X \setminus Z$.

As for the patching datum, consider $x \in V_{\alpha\beta}$ and a Stein open subset $U \subset V_{\alpha\beta}$ containing $x$. Then the two thickening $\prescript{k}{}{\mathbf{V}}_{\alpha}$ and $\prescript{k}{}{\mathbf{V}}_{\beta}$ can be identified via an isomorphism $\prescript{k}{}{\Psi}_{\alpha\beta,U} : \prescript{k}{}{\mathbf{V}}_{\alpha}|_{U} \rightarrow \prescript{k}{}{\mathbf{V}}_{\beta}|_{U}$ as log schemes over $\logsk{k}$ by \cite[Theorem 6.13]{Felten-Filip-Ruddat} (cf. also \cite[Lemma 2.15]{Gross-Siebert-logII}) as in \cite[\S 8]{chan2019geometry}. Taking an open covering $\mathcal{U}$ as in Notation \ref{not:open_stein_covers}, we have the following definition.

\begin{definition}\label{def:higher_order_patching_data_from_gross_siebert}
	For each $k \in \mathbb{N}$ and triple $(U_i; V_\alpha, V_\beta)$ with $U_i \subset V_{\alpha\beta}$, let $\prescript{k}{}{\Psi}_{\alpha\beta,i}$ be the following isomorphism of log schemes:
	$$
	\xymatrix@1{ \prescript{k}{}{\mathbf{V}}_{\alpha}^{\dagger}|_{U_i} \ar[rr]^{\prescript{k}{}{\Psi}_{\alpha\beta,i}} \ar[d]^{\pi_\alpha}& &  \prescript{k}{}{\mathbf{V}}_{\beta}^{\dagger}|_{U_i} \ar[d]^{\pi_{\beta}}\\
		\logsk{k}	\ar@{=}[rr]& & \logsk{k}.
	}
	$$
	The isomorphism of sheaves $\patchh{k}_{\alpha\beta,i} :  j_* (  \Theta_{\prescript{k}{}{\mathbf{V}}_{\alpha}^{\dagger}/\logsk{k}})|_{U_i} \rightarrow  j_* (  \Theta_{\prescript{k}{}{\mathbf{V}}_\beta^{\dagger}/\logsk{k}})|_{U_i}$ which appears in Definition \ref{def:patching_datum} is taken to be that induced by $\prescript{k}{}{\Psi}_{\alpha\beta,i}$. 
\end{definition}

The existence of the vector fields $\restah{k,l}_{\alpha\beta,i}$, $\patchijh{k}_{\alpha\beta,ij}$ and $\cocyobsh{k}_{\alpha\beta\gamma,i}$ in Definition \ref{def:patching_datum} follows from 
the analytic version of \cite[Theorem 2.11]{Gross-Siebert-logII} (we can use this theorem because the local models appearing in \cite{Felten-Filip-Ruddat} are the same as those in \cite{Gross-Siebert-logII}), which implies that
any log automorphism of the space $\prescript{k}{}{\mathbf{V}}_{\alpha}^{\dagger}|_{U_i} $ (resp. $\prescript{k}{}{\mathbf{V}}_{\alpha}^{\dagger}|_{U_{ij}}$) fixing $X|_{U_i}$ (resp. $X|_{U_{ij}}$) is obtained by exponentiating the action of a vector field in $\Theta_{\prescript{k}{}{\mathbf{V}}_{\alpha}^{\dagger}/\logsk{k}}(U_i)$ (resp. $\Theta_{\prescript{k}{}{\mathbf{V}}_{\alpha}^{\dagger}/\logsk{k}}(U_{ij})$). 

\subsection{Construction of abstract deformation data for a pair}\label{sec:pair_data_from_gross_siebert}
We now proceed to construct the abstract deformation data $\gla{}$ associated to the pair $(X,\mathfrak{C}^*)$. By passing to the determinant line bundle, we obtain the abstract deformation data $\bva{}$ associated to the pair $(X, \det \mathfrak{C}^*)$ and also a forgetful morphism $\morph{}: \gla{} \rightarrow \bva{}$, so that we can apply Theorem \ref{thm:main_theorem}.
We will need extra data about smoothing of $\mathfrak{C}^*$ along with $X$. Using the open covering $\{V_{\alpha}\}_\alpha$ and the thickenings $\prescript{k}{}{\mathbf{V}}_{\alpha}$'s in \S \ref{sec:data_from_gross_siebert} and letting $\prescript{k}{}{\mathcal{O}}_{\alpha}:=\mathcal{O}_{\prescript{k}{}{\mathbf{V}}_{\alpha}}$, we introduce the following definition:
\begin{definition}\label{def:geometric_lifting_data_for_pair}
	A {\em geometric lifting datum} $(\{\prescript{k}{}{\mathfrak{C}}_{\alpha}^*\}_{k,\alpha},\{\rest{k,l}_{\alpha}\}_{k\geq l,\alpha})$ of $\mathfrak{C}^*$ consists of
	\begin{itemize}
	\item
	for each $k \geq 0$ and $\alpha$, a perfect complex $\prescript{k}{}{\mathfrak{C}}_{\alpha}^*$ of $\prescript{k}{}{\mathcal{O}}_{\alpha}$-modules over $\prescript{k}{}{\mathbf{V}}_{\alpha}$, and
	\item
	for $k\geq l$ and each $\alpha$, a morphism $\rest{k,l}_{\alpha}: \prescript{k}{}{\mathfrak{C}}_{\alpha}^* \rightarrow \prescript{l}{}{\mathfrak{C}}_{\alpha}^*$ of complexes of $\prescript{k}{}{\mathcal{O}}_{\alpha}$-modules (here $\prescript{l}{}{\mathfrak{C}}_{\alpha}^*$ is treated as an $\prescript{k}{}{\mathcal{O}}_{\alpha}$-module via the natural homomorphism $\rest{k,l}_{\alpha}:\prescript{k}{}{\mathcal{O}}_{\alpha} \rightarrow \prescript{l}{}{\mathcal{O}}_{\alpha}$)
	\end{itemize}
	such that $\prescript{0}{}{\mathfrak{C}}_{\alpha}^* = \mathfrak{C}^*|_{V_{\alpha}}$, and $\rest{k,l}_{\alpha}$ is an isomorphism upon tensoring with $\cfrk{l}$ over $\cfrk{k}$.
\end{definition}

From a geometric lifting datum, we can construct an abstract deformation data $\gla{}$ as defined in Definition \ref{def:abstract_deformation_data} as follows.
On each $V_{\alpha}$, we define the sheaf $\gla{k}_{\alpha}^*$ of Lie algebras as the subsheaf of $\sla{k}_{\alpha} \times \shom_{\cfrk{k}}^*(\prescript{k}{}{\mathfrak{C}}^*_{\alpha},\prescript{k}{}{\mathfrak{C}}^*_{\alpha})$ whose stalk at $x$ is given by $\sla{k}_{\alpha,x}^*(\prescript{k}{}{\mathfrak{C}}^*_x)$ as in Definition \ref{def:Lie_algebra_of_pairs}, which comes with the natural morphism of sheaves of Lie algebras $\rest{k,l}_{\alpha} : \gla{k}_{\alpha}^* \rightarrow \gla{l}_{\alpha}^*$. There is a natural exact sequence of sheaves of Lie algebras
\begin{equation}\label{eqn:achnor_map_sequence}
0 \rightarrow \shom_{\prescript{k}{}{\mathcal{O}}_{\alpha}}^*(\prescript{k}{}{\mathfrak{C}}^*_{\alpha},\prescript{k}{}{\mathfrak{C}}^*_{\alpha}) \rightarrow \gla{k}_{\alpha}^* \rightarrow \sla{k}_{\alpha}^* \rightarrow 0.
\end{equation}
Conditions $(1),(2),(4)$ in Definition \ref{def:abstract_deformation_data} follow immediately from this exact sequence.
Condition $(3)$ does not hold since $\ker(\text{ad}) = \prescript{k}{}{\mathcal{O}}_{\alpha} \cdot \text{Id}_{\prescript{k}{}{\mathfrak{C}}^*_{\alpha}}$ (i.e., multiples of the identity endomorphism in $\shom_{(\prescript{k}{}{\mathcal{O}}_{\alpha})}^*(\prescript{k}{}{\mathfrak{C}}^*_{\alpha},\prescript{k}{}{\mathfrak{C}}^*_{\alpha})$). In this case one can easily check that $\text{ad}\left(\gla{k}_{\alpha}^*\right)$ satisfies conditions $(1),(2),(4)$ in Definition \ref{def:abstract_deformation_data}, so we obtain an abstract deformation datum $\gla{}$ under the weakened assumption as described in \S \ref{sec:weaken_assumption}. 
Now, we take $\bva{k}_{\alpha}$ to be the subsheaf of $\sla{k}_{\alpha} \times \shom_{\cfrk{k}}(\det \prescript{k}{}{\mathfrak{C}}^*_{\alpha},\det \prescript{k}{}{\mathfrak{C}}^*_{\alpha})$ whose stalk at $x$ is given by $\sla{k}_{\alpha,x}(\det \prescript{k}{}{\mathfrak{C}}^*_x)$ as described in \S \ref{sec:abstract_algebra_for_deformation_of_pairs}.

Next, using the geometric isomorphisms $\prescript{k}{}{\Psi}_{\alpha\beta,i} : \prescript{k}{}{\mathbf{V}}_{\alpha}^{\dagger}|_{U_i} \rightarrow \prescript{k}{}{\mathbf{V}}_{\beta}^{\dagger}|_{U_i}$ between log schemes over $\logsk{k}$ described in \S \ref{sec:data_from_gross_siebert}, we introduce the following notion of a geometric patching datum:

\begin{definition}\label{def:geometric_patching_data_of_pair}
	A {\em geometric patching datum} for $(\{\prescript{k}{}{\mathfrak{C}}_{\alpha}^*\}_{k,\alpha},\{\rest{k,l}_{\alpha}\}_{k\geq l,\alpha})$ consists of, for each $k \in \mathbb{N}$ and $(U_i; V_\alpha, V_\beta)$ with $U_i \subset V_{\alpha\beta} := V_{\alpha} \cap V_{\beta}$, an isomorphism $$\prescript{k}{}{\Xi}_{\alpha\beta,i} : \prescript{k}{}{\mathfrak{C}}_{\alpha}^*|_{U_i} \rightarrow \prescript{k}{}{\mathfrak{C}}_{\beta}^*|_{U_i}$$ of complexes of sheaves of modules which is compatible with the isomorphism $\prescript{k}{}{\Psi}_{\beta\alpha,i}^* : \prescript{k}{}{\mathcal{O}}_{\alpha}|_{U_i} \rightarrow \prescript{k}{}{\mathcal{O}}_{\beta}|_{U_i}$ of sheaves of rings, such that $\prescript{0}{}{\Xi}_{\alpha\beta,i} = \text{id}$. 
\end{definition}

Given a geometric patching datum $\{\prescript{k}{}{\Xi}_{\alpha\beta,i}\}$ of $(\{\prescript{k}{}{\mathfrak{C}}_{\alpha}^*\}_{k,\alpha},\{\rest{k,l}_{\alpha}\}_{k\geq l,\alpha})$, we get an isomorphism $\patchp{k}_{\alpha\beta,i} : \shom_{\cfrk{k}}(\prescript{k}{}{\mathfrak{C}}^*_{\alpha},\prescript{k}{}{\mathfrak{C}}^*_{\alpha})|_{U_i} \rightarrow \shom_{\cfrk{k}}(\prescript{k}{}{\mathfrak{C}}^*_{\beta},\prescript{k}{}{\mathfrak{C}}^*_{\beta})|_{U_i}$ of sheaves obtained via $\prescript{k}{}{\Xi}_{\alpha\beta,i}$'s. Putting this together with $\patchh{k}_{\alpha\beta,i} : \sla{k}_{\alpha}|_{U_i} \rightarrow \sla{k}_{\beta}|_{U_i}$, we obtain an isomorphism 
$$\patchp{k}_{\alpha\beta,i} : \sla{k}_{\alpha}|_{U_i} \times \shom_{\cfrk{k}}(\prescript{k}{}{\mathfrak{C}}^*_{\alpha},\prescript{k}{}{\mathfrak{C}}^*_{\alpha})|_{U_i} \rightarrow \sla{k}_{\beta}|_{U_i} \times \shom_{\cfrk{k}}(\prescript{k}{}{\mathfrak{C}}^*_{\beta},\prescript{k}{}{\mathfrak{C}}^*_{\beta})|_{U_i}$$
of sheaves over $V_{\alpha}$. It induces the corresponding isomorphism $\patchp{k}_{\alpha\beta,i} : \gla{k}_{\alpha}^*|_{U_i} \rightarrow \gla{k}_{\beta}^*|_{U_i}$ of sheaves of dgLa's over $V_{\alpha}$, producing a patching datum for $\gla{} = \{ \gla{k}_{\alpha} \}_{k,\alpha}$. By repeating the same construction for the geometric patching datum $\det \prescript{k}{}{\Xi}_{\alpha\beta,i}: \det \prescript{k}{}{\mathfrak{C}}_{\alpha}^*|_{U_i} \rightarrow \det \prescript{k}{}{\mathfrak{C}}_{\beta}^*|_{U_i} $, we obtain the isomorphism $\patch{k}_{\alpha\beta,i} : \bva{k}_{\alpha}^*|_{U_i} \rightarrow \bva{k}_{\beta}^*|_{U_i}$ of sheaves of dgLa's. 

We should indicate how to construct the elements $\restap{k,l}_{\alpha\beta,i}$'s, $\patchijp{k}_{\alpha\beta,i}$'s and $\cocyobsp{k}_{\alpha\beta\gamma,i}$'s. For $\cocyobsp{k}_{\alpha\beta\gamma,i}$'s, we observe that $\prescript{k}{}{\Xi}_{\gamma\alpha,i} \circ \prescript{k}{}{\Xi}_{\beta \gamma,i} \circ \prescript{k}{}{\Xi}_{\alpha\beta,i} : \prescript{k}{}{\mathfrak{C}}^*_{\alpha}|_{U_i} \rightarrow \prescript{k}{}{\mathfrak{C}}^*_{\alpha}|_{U_i}$ is an automorphism of complexes, which is compatible with the natural automorphism $\prescript{k}{}{\Psi}_{\alpha\gamma,i}^* \circ \prescript{k}{}{\Psi}_{ \gamma\beta,i}^* \circ \prescript{k}{}{\Psi}_{\beta\alpha,i}^* : \mathcal{O}_{\alpha}|_{U_i} \rightarrow \mathcal{O}_{\alpha}|_{U_i}$. Since $\prescript{0}{}{\Xi}_{\gamma\alpha,i} \circ \prescript{0}{}{\Xi}_{\beta \gamma,i} \circ \prescript{0}{}{\Xi}_{\alpha\beta,i} = \text{id}$, we can define $\cocyobsp{k}_{\alpha\beta\gamma,i} := \left(\cocyobs{k}_{\alpha\beta\gamma,i},\log(\prescript{k}{}{\Xi}_{\gamma\alpha,i} \circ \prescript{k}{}{\Xi}_{\beta \gamma,i} \circ \prescript{k}{}{\Xi}_{\alpha\beta,i}) \right) \in \gla{k}_{\alpha}^0(U_i)$.  Similarly, we define $\restap{k,l}_{\alpha\beta,i} := \left( \resta{k,l}_{\alpha\beta,i}, \log (\prescript{l}{}{\Xi}_{\beta \alpha,i} \circ (\prescript{k}{}{\Xi}_{\alpha\beta,i} \otimes_{(\cfrk{k})} \cfrk{l})) \right)$ and $\patchijp{k}_{\alpha\beta,ij}  := \left(\patchij{k}_{\alpha\beta,ij}, \log( \prescript{k}{}{\Xi}_{\beta \alpha,j}|_{U_{ij}} \circ \prescript{k}{}{\Xi}_{\alpha\beta,i}|_{U_{ij}}) \right)$. Therefore we obtain a patching data $\patchp{}$ for $\gla{}$. Repeating the above for $\det \prescript{k}{}{\Xi}_{\alpha\beta,i}$ produces the patching datum for $\bva{}$.
The trace map $\trace : \sla{k}_{\alpha}^*(\prescript{k}{}{\mathfrak{C}}^*_{\alpha}) \rightarrow \sla{k}_{\alpha}^0(\det \prescript{k}{}{\mathfrak{C}}^*_{\alpha})$ described in \S \ref{sec:abstract_algebra_for_deformation_of_pairs} gives $\morph{k}_{\alpha} : \gla{k}_{\alpha}^* \rightarrow  \bva{k}_{\alpha}$ which is a surjective morphism of sheaves of dgLa's. One can easily check conditions $(1)-(5)$ in Definition \ref{def:morphism_of_deformation_data}. We thus obtain a morphism $\morph{} : (\gla{},\patchp{}) \rightarrow (\bva{},\patch{})$.

Let $\mcsimplicialp{\bullet}(\simplex_{\bullet})$ and $\mcsimplicial{\bullet}(\simplex_{\bullet})$ be the Maurer-Cartan functors associated to the deformation data $(\gla{},\patchp{})$ and $(\bva{},\patch{})$ respectively.

\begin{corollary}\label{cor:smoothing_of_perfect_complexes}
	Let $X$ be a projective toroidal crossing space which is Calabi-Yau and with $\mathcal{T}^1_X := \mathcal{E}\text{xt}^1(\Omega_X,\mathcal{O}_X)$ globally generated. Let $\mathfrak{C}^*$ be a bounded complex of locally free sheaves (of finite rank) over $X$.
	Suppose the pair $(X,\mathfrak{C}^*)$ is equipped with a geometric lifting datum $(\{\prescript{k}{}{\mathfrak{C}}_{\alpha}^*\}_{k,\alpha},\{\rest{k,l}_{\alpha}\}_{k\geq l,\alpha})$ (Definition \ref{def:geometric_lifting_data_for_pair}) and a geometric patching datum $\{\prescript{k}{}{\Xi}_{\alpha\beta,i}^*\}_{k,\alpha\beta,i}$ (Definition \ref{def:geometric_patching_data_of_pair}).
	If $\text{Ext}^2(\mathfrak{C}^*,\mathfrak{C}^*)_0 = 0$, then the morphism $\morph{}:\mcsimplicialp{\bullet}(\simplex_{\bullet}) \rightarrow \mcsimplicial{\bullet}(\simplex_{\bullet})$ between Maurer-Cartan functors is smooth.
\end{corollary}

\begin{proof}
	The statement follows from Theorem \ref{thm:main_theorem} as we have $\mathbb{H}^2(X, \kla) = \text{Ext}^2(\mathfrak{C}^*,\mathfrak{C}^*)_0 = 0$.
	\end{proof}

In geometric situations, we are interested in the following notion concerning smoothing of geometric objects. 

\begin{definition}\label{def:formally_smoothable}
	We say that a variety $X$ is {\em formally smoothable} if for every $k \in \mathbb{N}$, there exists a flat family of log schemes $\pi : \prescript{k}{}{X}^{\dagger} \rightarrow \logsk{k}$ such that $X^{\dagger} = \prescript{0}{}{X}^{\dagger}$,  $\prescript{k+1}{}{X}^{\dagger} \times_{(\logsk{k+1})} \logsk{k} = \prescript{k}{}{X}^{\dagger}$ and $\prescript{k}{}{X}^{\dagger}|_{V_{\alpha}} = \prescript{k}{}{\mathbf{V}}^{\dagger}_{\alpha}$ for every $\alpha$ and $k \in \mathbb{N}$. 
	
	A pair $(X,\mathfrak{C}^*)$ is said to be {\em formally smoothable} if for every $k \in \mathbb{N}$, there exists a pair $(\prescript{k}{}{X},\prescript{k}{}{\mathfrak{C}}^*)$, where $\pi : \prescript{k}{}{X}^{\dagger} \rightarrow \logsk{k}$ is a flat family of log schemes as above and $\prescript{k}{}{\mathfrak{C}}^*$ is a bounded complex of locally free sheaves on $\prescript{k}{}{X}^{\dagger}$, such that $\prescript{k+1}{}{\mathfrak{C}}^*\otimes_{(\cfrk{k+1})} \cfrk{k}= \prescript{k}{}{\mathfrak{C}}^*$ and $\prescript{k}{}{\mathfrak{C}}^*|_{V_{\alpha}} = \prescript{k}{}{\mathfrak{C}}^*_{\alpha}$ for every $\alpha$ and $k \in \mathbb{N}$.
\end{definition}

Suppose that $X$ is formally smoothable. Then the isomorphism $\prescript{k}{}{\mathbf{V}}_\alpha^{\dagger}|_{V_{\alpha\beta}} \cong \prescript{k}{}{X}^{\dagger} \cong \prescript{k}{}{\mathbf{V}}_\beta^{\dagger}|_{V_{\alpha\beta}}$ induces an isomorphism $\sla{k}_{\alpha}|_{V_{\alpha \beta}} \cong \sla{k}_{\beta}|_{V_{\alpha\beta}}$, which can be further passed to the associated Thom-Whitney resolutions to give a set of compatible gluing morphisms $\glueh{k}_{\alpha\beta}' : \twc{k}_{\alpha;\alpha\beta} \rightarrow \twc{k}_{\beta;\alpha\beta}$ over $\simplex_0$.
Taking $\dbtwisth{k}_{\alpha}' = 0$ gives a set of compatible differentials with respect to $\glueh{k}_{\alpha\beta}'$'s over $\simplex_0$. 
Now $\prescript{k}{}{\hat{\varphi}}' = 0$ will be a Maurer-Cartan element over $\cfrk{k}$ with respect to $(\glueh{}',\dbtwisth{}')$ for each $k$. If we have another set of compatible gluing morphisms and differentials $(\glueh{},\dbtwisth{})$, then Lemma \ref{lem:gluing_data_contractible} gives a set of compatible gluing morphisms and differentials $(\glueh{}(\simplex_1),\dbtwisth{}(\simplex_1))$ connected to $(\glueh{}',\dbtwisth{}')$. Making use of Lemma \ref{lem:kan_complex_lemma}, we can then inductively construct a system of Maurer-Cartan elements $\prescript{k}{}{\hat{\varphi}}'s$ such that $\prescript{k+1}{}{\hat{\varphi}} \equiv \prescript{k}{}{\hat{\varphi}} \ (\text{mod $q^{k+1}$})$, with respect to $(\glueh{},\dbtwisth{})$. 
Similarly, if $(X,\mathfrak{C}^*)$ is formally smoothable, then for any set of compatible gluing morphisms and differentials $(\glue{},\dbtwist{})$, we have a system of Maurer-Cartan elements $\prescript{k}{}{\varphi}$'s such that $\prescript{k+1}{}{\varphi} \equiv \prescript{k}{}{\varphi} \ (\text{mod $q^{k+1}$})$.

\begin{proof}[Proof of Theorem \ref{thm:intro_main_theorem}]
	Under the assumption that the pair $(X,\det(\mathfrak{C}^*))$ is formally smoothable, we have a compatible system of Maurer-Cartan elements $\{\prescript{k}{}{\phi}\}_{k \in \mathbb{N}}$ by the above discussion. 
	Since $\morph{}:\mcsimplicialp{\bullet}(\simplex_{\bullet}) \rightarrow \mcsimplicial{\bullet}(\simplex_{\bullet})$ is smooth by Corollary \ref{cor:smoothing_of_perfect_complexes}, we can then inductively construct the desired compatible system of Maurer-Cartan elements. 
	
	To prove the second statement, we assume instead that $H^2(X, \mathcal{O}_X) = 0$.
	Then we can consider the Maurer-Cartan functor $\mcsimplicialh{\bullet}(\simplex_{\bullet})$ associated to the deformation datum $(\sla{},\patchh{})$. The anchor map described in Definition \ref{def:Lie_algebra_of_pairs} gives a map $\morphh{k}_\alpha: \sla{k}_{\alpha}^0(\det \prescript{k}{}{\mathfrak{C}}^*_{\alpha}) \rightarrow \sla{k}_{\alpha}$ for each $\alpha$, which patch together to give a morphism $\morphh{}: (\bva{},\patch{}) \rightarrow (\sla{},\patchh{})$ between abstract deformation data, thereby inducing a map $\morphh{} : \mcsimplicial{\bullet}(\simplex_{\bullet}) \rightarrow \mcsimplicialh{\bullet}(\simplex_{\bullet})$ between the associated Maurer-Cartan functors. Applying Theorem \ref{thm:main_theorem} to $\morphh{}$ and noting that $\mathbb{H}^2(X,\kla^*)  = H^2(X,\mathcal{O}_X) = 0$, we conclude that $\morphh{} : \mcsimplicial{\bullet}(\simplex_{\bullet})\rightarrow \mcsimplicialh{\bullet}(\simplex_{\bullet})$ is smooth. 
	Notice that the toroidal crossing space $X$ is formally smoothable by the results in \cite{chan2019geometry, Felten-Filip-Ruddat}. So by repeating the argument in the paragraph right before this proof, we can construct a compatible system of Maurer-Cartan elements $\{\prescript{k}{}{\phi}\}_{k\in \mathbb{N}}$, from which we can deduce the result by arguing as in the proof of the first statement.
	(By the discussion in \S \ref{sec:geometric_gluing_from_mc}, which is independent from the proof here, we can actually construct a formal smoothing of the pair $(X, \det \mathfrak{C}^*)$ from the system of Maurer-Cartan elements $\{\prescript{k}{}{\phi}\}_{k\in \mathbb{N}}$. So the second statement of Theorem \ref{thm:intro_main_theorem} is indeed a special case of the first one.)
\end{proof}

\subsection{Proof of Theorem \ref{thm:intro_main_application} -- geometric smoothing from Maurer-Cartan solutions}\label{sec:geometric_gluing_from_mc}

In this subsection, we explain how to apply the technique in \cite[\S 5.3]{chan2019geometry} to the compatible set of Maurer-Cartan elements constructed in Theorem \ref{thm:intro_main_theorem} to construct a geometric formal smoothing of the pair $(X, \prescript{}{}{\mathfrak{F}})$, where $X$ is as in Theorem \ref{thm:intro_main_application} and $\prescript{}{}{\mathfrak{F}}$ is a locally free sheaf on $X$ to be regarded as a complex concentrated in degree $0$.

First of all, if $\mathfrak{C}^*$ consists of just one locally free sheaf $\mathfrak{F}$ of rank $r$ concentrated at a fixed degree (which, without loss of generality, can be assumed to be zero), then we can construct a geometric lifting datum and a geometric patching datum for $\mathfrak{F}$ as follows. On a Stein open subset $V_{\alpha}$, we can trivialize $\mathfrak{F} = \bigoplus_{i=1}^r \mathcal{O}_{X}|_{V_{\alpha}} \cdot e_i$ and take $\mathfrak{F}_{\alpha} := \bigoplus_{i=1}^r \prescript{k}{}{\mathcal{O}}_{\alpha} \cdot e_i$. Then $\rest{k,l}_{\alpha}$ is simply given by identifying the frame $\{e_i\}$ in both trivializations and taking the restriction $\rest{k,l}_{\alpha} : \prescript{k}{}{\mathcal{O}}_{\alpha} \rightarrow \prescript{l}{}{\mathcal{O}}_{\alpha}$ on the coefficients. This gives a geometric lifting datum. Furthermore, one can write $\prescript{0}{}{\Xi}_{\alpha\beta,i} (e_s) = \sum_{t=1}^r A_{ts} e_t$ in terms of matrices $A_{ts} \in \mathcal{O}_{X}(U_i)$. We can then construct a patching datum by setting $\prescript{k}{}{\Xi}_{\alpha\beta,i}(e_s) := \sum_{t=1}^r \prescript{k}{}{A}_{ts} e_t$, where $\prescript{k}{}{A}_{ts} \in \prescript{k}{}{\mathcal{O}}_{\beta}(U_i)$ are elements lifting $A_{ts}$, and then extend linearly to make it compatible with the map $\prescript{k}{}{\Psi}_{\beta\alpha,i}^* : \prescript{k}{}{\mathcal{O}}_{\alpha}|_{U_i} \rightarrow \prescript{k}{}{\mathcal{O}}_{\beta}|_{U_i}$.  

To prove Theorem \ref{thm:intro_main_application}, we first assume that $\text{Ext}^2(\mathfrak{F},\mathfrak{F})_0  = 0$ and that the pair $(X, \det \mathfrak{F})$ is formally smoothable.
For the rest of this section, we restrict our attention to the $0$-simplex $\simplex_0$ (and will omit $\simplex_0$ from our notations). Let $\gluep{} = (\gluep{k}_{\alpha\beta})$ be a compatible gluing morphism for $(\gla{},\patchp{})$ over $\simplex_0$, which is given by $\iautop{k}_{\alpha\beta,i} $'s, $\sautop{k}_{\alpha\beta,i_0\cdots i_l}$'s and $\patchp{k}_{\alpha\beta,i}$'s as in Condition \ref{assum:induction_hypothesis} and Definition \ref{def:compatible_gluing_morphism}. It determines a compatible gluing morphism $\glueh{}$ for $(\sla{},\patchh{})$ with data $\iautoh{k}_{\alpha\beta,i} $'s, $\sautoh{k}_{\alpha\beta,i_0\cdots i_l}$'s and $\patchh{k}_{\alpha\beta,i}$'s from the natural morphism $ \morphc{k}_{\alpha} : \gla{k}_{\alpha} \rightarrow \sla{k}_{\alpha}$ in \eqref{eqn:achnor_map_sequence}). Furthermore, we fix a compatible differential $\dbtwistp{} = (\dbtwistp{k}_{\alpha})$ for $\gluep{}$ which determines $\dbtwisth{}  = \morphc{}(\dbtwistp{} )$ for $\glueh{}$. 
By Theorem \ref{thm:intro_main_theorem}, there is a Maurer-Cartan solution $\prescript{k}{}{\varphi} = (\prescript{k}{}{\varphi}_{\alpha})_{\alpha}$ of the dgLa $\polyv{k}(\gluep{},\mathfrak{F})$ for each $k \in \mathbb{N}$ such that $\rest{k,l}(\prescript{k}{}{\varphi}) = \prescript{l}{}{\varphi}$. 

Similar to Definitions \ref{def:thom_whitney_general} and \ref{def:cech_thom_whitney_complex}, we define
\begin{align*}
	\two{k}^{*}_{\alpha}(W) & := \left\{  (\phi_{i_0 \cdots i_l})_{(i_0,\dots,  i_l) \in \mathcal{I}} \mid \phi_{i_0 \cdots i_l} \in \mathcal{A}^*(\simplex_l) \otimes \prescript{k}{}{\mathcal{O}}_{\alpha_j} (U_{i_0\cdots i_l}), \mathtt{d}_{j,l}^* (\phi_{i_0 \cdots i_l}) = \phi_{i_0 \cdots \hat{i}_j \cdots i_l} |_{U_{i_0 \cdots i_l}} \right\}, \\
	\tf{k}^{*}_{\alpha}(W) & := \left\{ (\phi_{i_0 \cdots i_l})_{(i_0,\dots,  i_l) \in \mathcal{I}} \mid \phi_{i_0 \cdots i_l} \in \mathcal{A}^*(\simplex_l)  \otimes \prescript{k}{}{\mathfrak{F}}_{\alpha}(U_{i_0\cdots i_l}), \mathtt{d}_{j,l}^* (\phi_{i_0 \cdots i_l}) = \phi_{i_0 \cdots \hat{i}_j \cdots i_l} |_{U_{i_0 \cdots i_l}} \right\},
\end{align*}
where $\mathcal{I} = \{ (i_0,\cdots,i_l) \mid U_{i_j} \subset W \}$ is a covering for an open subset $W \subset V_{\alpha}$.
For $W \subset W'$, we have the restriction map $\restmap_{W,W'}$ defined by
$$
\restmap_{W,W'} \Big( (\phi_{I})_{I \in \mathcal{I}} \Big) = (\phi_I)_{I \in \mathcal{I}'},
$$
where $\mathcal{I}' = \{ (i_0,\dots,i_l) \in \mathcal{I} \mid U_{i_j} \subset W'\}$ as in Notation \ref{not:local_thom_whitney_complex}. We equip $\two{k}^{*}_{\alpha}$ and $\tf{k}^{*}_{\alpha}$ with the operator $\prescript{k}{}{\pdb}_{\alpha} + (\dbtwistp{k}_{\alpha} + \prescript{k}{}{\varphi}_{\alpha}) \cdot $ where the action is defined via the natural actions of $\gla{k}_{\alpha}$ on $\prescript{k}{}{\mathcal{O}}_{\alpha}$ and $\prescript{k}{}{\mathfrak{F}}_{\alpha}$ respectively. This turns $\two{k}^{*}_{\alpha}$ into a presheaf of dga's over $V_{\alpha}$ and $\tf{k}^{*}_{\alpha}$ into a presheaf of dg modules over $\two{k}^{*}_{\alpha}$. 

We can define a gluing morphism
$$\glueh{k}_{\alpha\beta} (\phi_{i_0\cdots i_l}) = \exp(\sautoh{k}_{\alpha\beta,i_0\cdots i_l}  \cdot ) \circ  \exp(\iautoh{k}_{\alpha\beta,i} \cdot ) \circ \prescript{k}{}{\Psi}_{\beta\alpha,i_0}^* (\phi_{i_0\cdots i_l}) $$
for $(\phi_{i_0\cdots i_l}) \in \two{k}_{\alpha}^{*}|_{V_{\alpha\beta}}$ using this action. The compatibility of the gluing morphism $\glueh{}$ allows us to glue these presheaves together to obtain a presheaf $\prescript{k}{}{\mathcal{O}}^*(\glueh{})$ of dga's. Similarly, by letting $$\gluep{k}_{\alpha\beta} (\phi_{i_0\cdots i_l}) = \exp(\sautop{k}_{\alpha\beta,i_0\cdots i_l} \cdot ) \circ  \exp(\iautop{k}_{\alpha\beta,i}  \cdot ) \circ \prescript{k}{}{\Xi}_{\alpha\beta,i_0} (\phi_{i_0\cdots i_l}),$$
we obtain a gluing homomorphism $\gluep{k}_{\alpha\beta} : \tf{k}_{\alpha}^{*}|_{V_{\alpha\beta}} \rightarrow \tf{k}_{\beta}^{*}|_{V_{\alpha\beta}}$. This gives complexes $\prescript{k}{}{\mathcal{O}}^*(\glue{})$ and $\prescript{k}{}{\mathfrak{F}}^{*}(\gluep{})$, where the operator $\prescript{k}{}{\pdb} +  \prescript{k}{}{\varphi}$ is defined by gluing the local operators $\prescript{k}{}{\pdb}_{\alpha}+ \dbtwistp{k}_{\alpha} +  \prescript{k}{}{\varphi}_{\alpha} \cdot $ together. Therefore we obtain a global presheaf $\prescript{k}{}{\mathfrak{F}}^{*}(\gluep{})$ of dg modules over the global presheaf $\prescript{k}{}{\mathcal{O}}^*(\glue{})$ of dga's. The above construction is compatible with the natural maps $\rest{k,l} : \prescript{k}{}{\mathcal{O}}^*(\glue{}) \rightarrow \prescript{l}{}{\mathcal{O}}^*(\glue{})$ and $\rest{k,l} : \prescript{k}{}{\mathfrak{F}}^*(\gluep{}) \rightarrow \prescript{l}{}{\mathfrak{F}}^*(\gluep{})$ induced by $\rest{k,l}_{\alpha}$ locally.

On the Stein open subset $V_{\alpha}$, we notice that $\prescript{k}{}{\pdb}_{\alpha}+ \dbtwistp{k}_{\alpha} +  \prescript{k}{}{\varphi}_{\alpha} \cdot $ is gauge equivalent to $\prescript{k}{}{\pdb}_{\alpha}$ via an element $\sauto{k}_{\alpha} \in \twc{k}^{0,0}_{\alpha}$. Conjugating with the automorphism $\exp(\prescript{k}{}{\vartheta}_\alpha \cdot)$ acting on 
$\tf{k}^{*}_{\alpha}$, and similarly on $\two{k}^*_{\alpha}$, for each $\alpha$, we obtain a gluing morphism $\prescript{k}{}{\mathbf{g}}_{\alpha\beta} : \tf{k}^{*}_{\alpha}|_{V_{\alpha\beta}} \rightarrow  \tf{k}^{*}_{\beta}|_{V_{\alpha\beta}}$, which fit into the following commutative diagram:
$$
\xymatrix@1{
	\tf{k}^{*}_{\alpha}|_{V_{\alpha\beta}} \ar[rr]^{\gluep{k}_{\alpha\beta}} \ar[d]^{\exp(\prescript{k}{}{\vartheta}_\alpha \cdot)} & &  \tf{k}^{*}_{\beta}|_{V_{\alpha\beta}}
	\ar[d]^{\exp(\prescript{k}{}{\vartheta}_\beta \cdot)}\\
	( \tf{k}^{*}_{\alpha}|_{V_{\alpha\beta}}, \prescript{k}{}{\pdb}_{\alpha}) \ar[rr]^{\prescript{k}{}{\mathbf{g}}_{\alpha\beta}} & &  (\tf{k}^{*}_{\beta}|_{V_{\alpha\beta}},\prescript{k}{}{\pdb}_{\beta})};
$$
here we emphasize that $\prescript{k}{}{\mathbf{g}}_{\alpha\beta}$ identifies the differentials $\prescript{k}{}{\pdb}_{\alpha}$ and $\prescript{k}{}{\pdb}_{\beta}$.

Now for any Stein open subset $U \subset V_{\alpha}$, we define
$$\prescript{k}{}{\mathfrak{F}}(U)  := H^0(\tf{k}^{*}_{\alpha}(U),\prescript{k}{}{\pdb}_\alpha).$$
Then the maps $\prescript{k}{}{\mathbf{g}}_{\alpha\beta} : \prescript{k}{}{\mathfrak{F}}_{\alpha}|_{V_{\alpha\beta}} \rightarrow \prescript{k}{}{\mathfrak{F}}_{\beta}|_{V_{\alpha\beta}}$ give isomorphisms of sheaves which satisfy the cocycle condition. Therefore we obtain a global sheaf $\prescript{k}{}{\mathfrak{F}}$. Similarly, the isomorphisms $\prescript{k}{}{\mathbf{g}}_{\alpha\beta} : \prescript{k}{}{\mathcal{O}}_{\alpha}|_{V_{\alpha\beta}} \rightarrow \prescript{k}{}{\mathcal{O}}_{\beta}|_{V_{\alpha\beta}}$ produce the global structure sheaf $\prescript{k}{}{\mathcal{O}}$, which defines a $k$-th order thickening of $X$. The $\prescript{k}{}{\mathcal{O}}$-module $\prescript{k}{}{\mathfrak{F}}$ is then a $k$-th order thickening of $\mathfrak{F}$.
In conclusion, we obtain a $k$-th order thickening $(\prescript{k}{}{\mathcal{O}}, \prescript{k}{}{\mathfrak{F}})$ of the pair $(X,\mathfrak{F})$ over $\cfrk{k} = \comp[q]/(q^{k+1})$ for each $k \in \mathbb{N}$ such that $\prescript{k+1}{}{\mathcal{O}} \otimes_{(\cfrk{k+1})} \cfrk{k}= \prescript{k}{}{\mathcal{O}}$ and $\prescript{k+1}{}{\mathfrak{F}} \otimes_{(\cfrk{k+1})} \cfrk{k}= \prescript{k}{}{\mathfrak{F}}$, so that the limit $\varprojlim_k (\prescript{k}{}{\mathcal{O}}, \prescript{k}{}{\mathfrak{F}})$ gives a formal smoothing of $(X,\mathfrak{F})$.

If we assume that $H^2(X,\mathcal{O}_X) = 0$, instead of formal smoothability of the pair $(X, \det \mathfrak{F})$, then (proof of) Theorem \ref{thm:intro_main_theorem} still gives the required compatible set $\{\prescript{k}{}{\varphi}\}_{k \in \mathbb{N}}$ of Maurer-Cartan elements, so
the pair $(X,\mathfrak{F})$ is again formally smoothable. 

\bibliographystyle{amsplain}
\bibliography{geometry}

\providecommand{\bysame}{\leavevmode\hbox to3em{\hrulefill}\thinspace}
\providecommand{\MR}{\relax\ifhmode\unskip\space\fi MR }
\providecommand{\MRhref}[2]{%
  \href{http://www.ams.org/mathscinet-getitem?mr=#1}{#2}
}
\providecommand{\href}[2]{#2}
\begin{thebibliography}{10}

\bibitem{bogomolov1978hamiltonian}
F.~A. Bogomolov, \emph{Hamiltonian {K}\"{a}hlerian manifolds}, Dokl. Akad. Nauk
  SSSR \textbf{243} (1978), no.~5, 1101--1104.

\bibitem{chan2019geometry}
K.~Chan, N.~C. Leung, and Z.~N. Ma, \emph{Geometry of the {M}aurer-{C}artan
  equation near degenerate {C}alabi-{Y}au varieties}, preprint,
  \href{http://arxiv.org/abs/1902.11174}{arXiv:1902.11174}.

\bibitem{chan2020tropical}
K.~Chan, Z.~N. Ma, and Y.-H. Suen, \emph{Tropical {L}agrangian multi-sections
  and smoothing of locally free sheaves over degenerate {C}alabi-{Y}au
  surfaces}, preprint,
  \href{http://arxiv.org/abs/2004.00523}{arXiv:2004.00523}.

\bibitem{chan2016differential}
K.~Chan and Y.-H. Suen, \emph{A differential-geometric approach to deformations
  of pairs {$(X,E)$}}, Complex Manifolds \textbf{3} (2016), no.~1, 16--40.

\bibitem{demailly1997complex}
J.-P. Demailly, \emph{Complex analytic and differential geometry}, 2012,
  \href{https://www-fourier.ujf-grenoble.fr/~demailly/manuscripts/agbook.pdf}{https://www-fourier.ujf-grenoble.fr/~demailly/manuscripts/agbook.pdf}.

\bibitem{dupont1976simplicial}
J.~L. Dupont, \emph{Simplicial de {R}ham cohomology and characteristic classes
  of flat bundles}, Topology \textbf{15} (1976), no.~3, 233--245.

\bibitem{Felten}
S.~Felten, \emph{Log smooth deformation theory via {G}erstenhaber algebras},
  preprint, \href{http://arxiv.org/abs/2001.02995}{arXiv:2001.02995}.

\bibitem{Felten-Filip-Ruddat}
S.~Felten, M.~Filip, and H.~Ruddat, \emph{Smoothing toroidal crossing spaces},
  preprint, \href{http://arxiv.org/abs/1908.11235}{arXiv:1908.11235}.

\bibitem{fiorenza2012differential}
D.~Fiorenza, D.~Iacono, and E.~Martinengo, \emph{Differential graded {L}ie
  algebras controlling infinitesimal deformations of coherent sheaves}, J. Eur.
  Math. Soc. (JEMS) \textbf{14} (2012), no.~2, 521--540.

\bibitem{fiorenza2007structures}
D.~Fiorenza and M.~Manetti, \emph{{$L_\infty$} structures on mapping cones},
  Algebra Number Theory \textbf{1} (2007), no.~3, 301--330.

\bibitem{fiorenza2012formality}
\bysame, \emph{Formality of {K}oszul brackets and deformations of holomorphic
  {P}oisson manifolds}, Homology, Homotopy and Applications \textbf{14} (2012),
  no.~2, 63--75.

\bibitem{fiorenza2012cosimplicial}
D.~Fiorenza, M.~Manetti, and E.~Martinengo, \emph{Cosimplicial {DGLA}s in
  deformation theory}, Comm. Algebra \textbf{40} (2012), no.~6, 2243--2260.

\bibitem{Friedman83}
R.~Friedman, \emph{Global smoothings of varieties with normal crossings}, Ann.
  of Math. (2) \textbf{118} (1983), no.~1, 75--114.

\bibitem{fukaya05}
K.~Fukaya, \emph{Multivalued {M}orse theory, asymptotic analysis and mirror
  symmetry}, Graphs and patterns in mathematics and theoretical physics, Proc.
  Sympos. Pure Math., vol.~73, Amer. Math. Soc., Providence, RI, 2005,
  pp.~205--278.

\bibitem{getzler2009lie}
E.~Getzler, \emph{Lie theory for nilpotent {$L_\infty$}-algebras}, Ann. of
  Math. (2) \textbf{170} (2009), no.~1, 271--301.

\bibitem{griffiths1981rational}
P.~Griffiths and J.~Morgan, \emph{Rational homotopy theory and differential
  forms}, Progress in Mathematics, vol.~16, Birkh\"{a}user, Boston, Mass.,
  1981.

\bibitem{GHKS2016theta}
M.~Gross, P.~Hacking, and B.~Siebert, \emph{Theta functions on varieties with
  effective anti-canonical class}, Mem. Amer. Math. Soc., to appear,
  \href{http://arxiv.org/abs/1601.07081}{arXiv:1601.07081}.

\bibitem{Gross-Siebert-logI}
M.~Gross and B.~Siebert, \emph{Mirror symmetry via logarithmic degeneration
  data. {I}}, J. Differential Geom. \textbf{72} (2006), no.~2, 169--338.

\bibitem{Gross-Siebert-logII}
\bysame, \emph{Mirror symmetry via logarithmic degeneration data, {II}}, J.
  Algebraic Geom. \textbf{19} (2010), no.~4, 679--780.

\bibitem{gross2011real}
\bysame, \emph{From real affine geometry to complex geometry}, Ann. of Math.
  (2) \textbf{174} (2011), no.~3, 1301--1428.

\bibitem{gross2016theta}
\bysame, \emph{Theta functions and mirror symmetry}, Surveys in differential
  geometry 2016. {A}dvances in geometry and mathematical physics, Surv. Differ.
  Geom., vol.~21, Int. Press, Somerville, MA, 2016, pp.~95--138.

\bibitem{hinich1996descent}
V.~Hinich, \emph{Descent of {D}eligne groupoids}, Internat. Math. Res. Notices
  (1997), no.~5, 223--239.

\bibitem{huang1995joint}
L.~Huang, \emph{On joint moduli spaces}, Math. Ann. \textbf{302} (1995), no.~1,
  61--79.

\bibitem{huybrechts2010deformation}
D.~Huybrechts and R.~P. Thomas, \emph{Deformation-obstruction theory for
  complexes via {A}tiyah and {K}odaira-{S}pencer classes}, Math. Ann.
  \textbf{346} (2010), no.~3, 545--569.

\bibitem{iacono2007differential}
D.~Iacono, \emph{Differential graded {L}ie algebras and deformations of
  holomorphic maps}, PhD Thesis, Rome, (2006),
  \href{http://arxiv.org/abs/math/0701091}{arXiv:math/0701091}.

\bibitem{iacono2014deformations}
\bysame, \emph{Deformations and obstructions of pairs {$(X,D)$}}, Int. Math.
  Res. Not. IMRN (2015), no.~19, 9660--9695.

\bibitem{iacono2018deformations}
D.~Iacono and M.~Manetti, \emph{On deformations of pairs (manifold, coherent
  sheaf)}, Canad. J. Math. \textbf{71} (2019), no.~5, 1209--1241.

\bibitem{KKP08}
L.~Katzarkov, M.~Kontsevich, and T.~Pantev, \emph{Hodge theoretic aspects of
  mirror symmetry}, From {H}odge theory to integrability and {TQFT}
  tt*-geometry, Proc. Sympos. Pure Math., vol.~78, Amer. Math. Soc.,
  Providence, RI, 2008, pp.~87--174.

\bibitem{kawamata1994logarithmic}
Y.~Kawamata and Y.~Namikawa, \emph{Logarithmic deformations of normal crossing
  varieties and smoothing of degenerate {C}alabi-{Y}au varieties}, Invent.
  Math. \textbf{118} (1994), no.~3, 395--409.

\bibitem{kontsevichgeneralized}
M.~Kontsevich, \emph{Generalized {T}ian-{T}odorov theorems}, talk on Kinosaki
  conference 2008.

\bibitem{kontsevich-soibelman04}
M.~Kontsevich and Y.~Soibelman, \emph{Affine structures and non-{A}rchimedean
  analytic spaces}, The unity of mathematics, Progr. Math., vol. 244,
  Birkh\"auser Boston, Boston, MA, 2006, pp.~321--385.

\bibitem{LYZ}
N.~C. Leung, S.-T. Yau, and E.~Zaslow, \emph{From special {L}agrangian to
  {H}ermitian-{Y}ang-{M}ills via {F}ourier-{M}ukai transform}, Adv. Theor.
  Math. Phys. \textbf{4} (2000), no.~6, 1319--1341.

\bibitem{li2008deformation}
S.~Li, \emph{On the deformation theory of pair $({X}, {E})$}, preprint,
  \href{http://arxiv.org/abs/0809.0344}{arXiv:0809.0344}.

\bibitem{Lurie}
J.~Lurie, \emph{Moduli problems for ring spectra}, Proceedings of the
  {I}nternational {C}ongress of {M}athematicians. {V}olume {II}, Hindustan Book
  Agency, New Delhi, 2010, pp.~1099--1125.

\bibitem{manetti2007lie}
M.~Manetti, \emph{Lie description of higher obstructions to deforming
  submanifolds}, Ann. Sc. Norm. Super. Pisa Cl. Sci. (5) \textbf{6} (2007),
  no.~4, 631--659.

\bibitem{manetti2005differential}
\bysame, \emph{Differential graded {L}ie algebras and formal deformation
  theory}, Algebraic geometry---{S}eattle 2005. {P}art 2, Proc. Sympos. Pure
  Math., vol.~80, Amer. Math. Soc., Providence, RI, 2009, pp.~785--810.

\bibitem{manetti2015some}
\bysame, \emph{On some formality criteria for {DG}-{L}ie algebras}, J. Algebra
  \textbf{438} (2015), 90--118.

\bibitem{martinengohigher}
E.~Martinengo, \emph{Higher brackets and moduli space of vector bundles}, Ph.D.
  thesis, Sapienza Universit\`a di Roma, 2009.

\bibitem{pridham2010unifying}
J.~P. Pridham, \emph{Unifying derived deformation theories}, Adv. Math.
  \textbf{224} (2010), no.~3, 772--826.

\bibitem{Siebert-Schroer}
S.~Schr\"oer and B.~Siebert, \emph{Toroidal crossings and logarithmic
  structures}, Adv. Math. \textbf{202} (2006), no.~1, 189--231.

\bibitem{sernesi2007deformations}
E.~Sernesi, \emph{Deformations of algebraic schemes}, Grundlehren der
  Mathematischen Wissenschaften [Fundamental Principles of Mathematical
  Sciences], vol. 334, Springer-Verlag, Berlin, 2006.

\bibitem{terilla2008smoothness}
J.~Terilla, \emph{Smoothness theorem for differential {BV} algebras}, J. Topol.
  \textbf{1} (2008), no.~3, 693--702.

\bibitem{tian1987smoothness}
G.~Tian, \emph{Smoothness of the universal deformation space of compact
  {C}alabi-{Y}au manifolds and its {P}etersson-{W}eil metric}, Mathematical
  aspects of string theory ({S}an {D}iego, {C}alif., 1986), Adv. Ser. Math.
  Phys., vol.~1, World Sci. Publishing, Singapore, 1987, pp.~629--646.

\bibitem{todorov1989weil}
A.~N. Todorov, \emph{The {W}eil-{P}etersson geometry of the moduli space of
  {${\rm SU}(n\geq 3)$} ({C}alabi-{Y}au) manifolds. {I}}, Comm. Math. Phys.
  \textbf{126} (1989), no.~2, 325--346.

\bibitem{wang2012deformations}
J.~Wang, \emph{Deformations of pairs {$(X,L)$} when {$X$} is singular}, Proc.
  Amer. Math. Soc. \textbf{140} (2012), no.~9, 2953--2966.

\bibitem{weibel1995introduction}
C.~A. Weibel, \emph{An introduction to homological algebra}, Cambridge Studies
  in Advanced Mathematics, vol.~38, Cambridge University Press, Cambridge,
  1994.

\bibitem{whitney2012geometric}
H.~Whitney, \emph{Geometric integration theory}, Princeton University Press,
  Princeton, N. J., 1957.

\end{thebibliography}

\end{document}